\documentclass[12pt]{article}
\pdfoutput=1

\usepackage{amsmath, amsthm, amssymb}
\usepackage{enumerate}
\usepackage{ifpdf}
\ifpdf
\usepackage[pdftex]{graphicx}
\else
\usepackage[dvips]{graphicx}
\fi
\usepackage{tikz}
 	 \usetikzlibrary{arrows,backgrounds}
\usepackage[all]{xy}

\input xy
\xyoption{all}

\usepackage[pdftex,plainpages=false,hypertexnames=false,pdfpagelabels]{hyperref}
\newcommand{\arxiv}[1]{\href{http://arxiv.org/abs/#1}{\tt arXiv:\nolinkurl{#1}}}
\newcommand{\arXiv}[1]{\href{http://arxiv.org/abs/#1}{\tt arXiv:\nolinkurl{#1}}}
\newcommand{\doi}[1]{\href{http://dx.doi.org/#1}{{\tt DOI:#1}}}

\newcommand{\googlebooks}[1]{(preview at \href{http://books.google.com/books?id=#1}{google books})}

\usepackage{xcolor}
\definecolor{dark-red}{rgb}{0.7,0.25,0.25}
\definecolor{dark-blue}{rgb}{0.15,0.15,0.55}
\definecolor{medium-blue}{rgb}{0,0,.8}
\definecolor{DarkGreen}{RGB}{0,150,0}
\hypersetup{
   colorlinks, linkcolor={purple},
   citecolor={medium-blue}, urlcolor={medium-blue}
}

\usepackage{longtable}

\theoremstyle{plain}
\newtheorem{thm}{Theorem}[section]
\newtheorem*{thm*}{Theorem}
\newtheorem{cor}[thm]{Corollary}
\newtheorem*{cor*}{Corollary}
\newtheorem{lem}[thm]{Lemma}

\newtheorem{prop}[thm]{Proposition}

\newtheorem*{quest*}{Question}
\newtheorem*{claim*}{Claim}

\theoremstyle{definition}
\newtheorem{defn}[thm]{Definition}

\newtheorem{note}[thm]{Note}
\newtheorem{exs}[thm]{Examples}
\newtheorem{ex}[thm]{Example}
\newtheorem{rem}[thm]{Remark}

\newtheorem{rems}[thm]{Remarks}

\DeclareMathOperator{\coeff}{coeff}
\DeclareMathOperator{\Out}{Out}
\DeclareMathOperator{\spann}{span}
\DeclareMathOperator{\Stab}{Stab}
\DeclareMathOperator{\Tr}{Tr}

\newcommand{\D}{\displaystyle}
\newcommand{\comment}[1]{}

\newcommand{\be}{\begin{enumerate}[(1)]}
\newcommand{\ee}{\end{enumerate}}

\newcommand{\Z}{\mathbb{Z}}

\newcommand{\I}{\infty}
\newcommand{\set}[2]{\left\{#1 \middle| #2\right\}}

\def\semicolon{;}
\def\applytolist#1{
    \expandafter\def\csname multi#1\endcsname##1{
        \def\multiack{##1}\ifx\multiack\semicolon
            \def\next{\relax}
        \else
            \csname #1\endcsname{##1}
            \def\next{\csname multi#1\endcsname}
        \fi
        \next}
    \csname multi#1\endcsname}

\def\calc#1{\expandafter\def\csname c#1\endcsname{{\mathcal #1}}}
\applytolist{calc}QWERTYUIOPLKJHGFDSAZXCVBNM;
\def\bbc#1{\expandafter\def\csname bb#1\endcsname{{\mathbb #1}}}
\applytolist{bbc}QWERTYUIOPLKJHGFDSAZXCVBNM;
\def\bfc#1{\expandafter\def\csname bf#1\endcsname{{\mathbf #1}}}
\applytolist{bfc}QWERTYUIOPLKJHGFDSAZXCVBNM;
\def\sfc#1{\expandafter\def\csname s#1\endcsname{{\sf #1}}}
\applytolist{sfc}QWERTYUIOPLKJHGFDSAZXCVBNM;

\newcommand{\jw}[1]{f^{(#1)}}

\newcommand{\noshow}[1]{}
\newcommand{\MR}[1]{}
\newcommand{\TL}{\cT\hspace{-.08cm}\cL}

\usetikzlibrary{calc}
\tikzstyle{shaded}=[fill=red!10!blue!20!gray!30!white]
\tikzstyle{unshaded}=[fill=white]
\tikzstyle{empty box}=[circle, draw, thick, fill=white, opaque, inner sep=2mm]
\tikzstyle{annular}=[scale=.7, inner sep=1mm, baseline]
\tikzstyle{rectangular}=[scale=.75, inner sep=1mm, baseline=-.1cm]

\newcommand{\nbox}[5]{
	\draw[thick, #1] ($#2+(-.4,-.4)+(-#3,0)$) -- ($#2+(-.4,.4)+(-#3,0)$) -- ($#2+(.4,.4)+(#4,0)$) -- ($#2+(.4,-.4)+(#4,0)$) -- ($#2+(-.4,-.4)+(-#3,0)$); 
	\coordinate (a) at ($#2+(-#3,0)$);
	\coordinate (b) at ($#2+(#4,0)$);
	\node at ($1/2*(a)+1/2*(b)$) {$#5$};
}

\makeatletter
\newcommand{\hashdef}[2]{\@namedef{#1}{#2}}
\newcommand{\hashlookup}[1]{\@nameuse{#1}}
\makeatother

\newcommand{\pathtographs}{diagrams/graphs/}

\hashdef{bwd1v1p1p1duals1v1x2x3}{CDEBF7CA2995B855}%
\hashdef{bwd1v1p1p1duals1v1x3x2}{30F84974997CC5A4}%
\hashdef{bwd1v1p1p1v0x0x1p0x0x1duals1v2x1x3}{3CA5711FC31ECF6F}%
\hashdef{bwd1v1p1p1v0x0x1p0x0x1v1x0p0x1v1x0p1x1p0x1v1x0x0p0x0x1v1x0p1x1p0x1v1x0x1v1v1p1p1duals1v2x1x3v2x1v1x2v1v1x2x3}{417CCB88C2AE577D}%
\hashdef{bwd1v1p1p1v0x0x1p0x0x1v1x0p1x0p1x0p0x1p0x1p0x1duals1v2x1x3v4x5x6x1x2x3}{4EAB57424A63B4C0}%
\hashdef{bwd1v1p1p1v0x0x1p0x0x1v1x1v1v1p1p1duals1v2x1x3v1v1x2x3}{BBA51A83A216D5C1}%
\hashdef{bwd1v1p1p1v0x0x1p0x0x1v1x1v1v1p1p1duals1v2x1x3v1v1x3x2}{46E18E4FBE85D50A}%
\hashdef{bwd1v1p1p1v0x0x2duals1v2x1x3}{A33AC303014C8A8B}%
\hashdef{bwd1v1p1p1v0x1x0p0x0x1duals1v1x3x2}{88D67A817E40ED00}%
\hashdef{bwd1v1p1p1v0x1x0p0x0x1v1x0p0x1p1x1v1x1x0v1p1duals1v1x3x2v1x2x3v1x2}{FE9E1B2DE018603B}%
\hashdef{bwd1v1p1p1v0x1x0p0x0x1v1x0p1x0p0x1p0x1v0x1x0x0p1x0x0x1p0x0x1x0v1x0x0p0x1x0p1x0x1p0x0x1v0x1x0x0p1x0x0x0p0x0x0x1v1x1x0p0x1x0p1x0x1p0x0x1v0x1x0x1v1p1duals1v1x3x2v4x2x3x1v4x3x2x1v1x2x3x4v1x2}{29B31945C06E2945}%
\hashdef{bwd1v1p1p1v0x1x0p0x0x1v1x0p1x0p1x0p0x1p0x1p0x1v0x0x1x0x0x1v1p1duals1v1x3x2v1x2x6x4x5x3v1x2}{9B1A4B4A85A23E9D}%
\hashdef{bwd1v1p1p1v0x1x0p0x0x1v1x0p1x0p1x1p0x1p0x1duals1v1x3x2v4x5x3x1x2}{8493AF1C6B83B131}%
\hashdef{bwd1v1p1p1v0x1x1v1p1duals1v1x3x2v1x2}{DB1CA226ACCA8804}%
\hashdef{bwd1v1p1p1v1x0x0duals1v1x2x3}{62D56813977DC10B}%
\hashdef{bwd1v1p1v1x0p0x1p0x1duals1v1x2}{4C1337B54CECE742}%
\hashdef{bwd1v1p1v1x0p1x0p0x1duals1v1x2}{E787E3120C409620}%
\hashdef{bwd1v1v1p1v0x1p1x0p1x0v1x0x0p1x0x0p0x1x0p0x0x1v1x0x0x0p0x1x0x0p0x0x1x0p0x0x1x0p0x0x0x1v1x0x0x0x0p0x1x0x0x0p0x0x0x1x0p0x0x0x0x1p0x0x0x0x1duals1v1v2x1x3v1x3x2x5x4}{CBDEA5397ABE4B4F}%
\hashdef{bwd1v1v1p1v1x0p0x1p0x1duals1v1v2x1x3}{745112D926A4C709}%
\hashdef{bwd1v1v1p1v1x0p0x1p0x1v0x1x0p1x0x1p0x0x1v0x1x0p1x0x1v1x0duals1v1v2x1x3v2x1}{43A4E7ED1401C97E}%
\hashdef{bwd1v1v1p1v1x0p1x0p0x1p0x1v0x1x1x0duals1v1v1x3x2x4}{CC3B89144EA94D74}%
\hashdef{bwd1v1v1p1v1x0p1x0p0x1v1x0x0p0x0x1p0x1x0p0x0x1v1x0x0x0p0x1x0x0p0x0x1x0p0x0x1x0p0x0x0x1v0x0x1x0x0p1x0x0x0x0p0x0x0x0x1p1x0x0x0x0p0x1x0x0x0duals1v1v1x3x2v3x4x1x2x5}{E86096D7836AA7B0}%
\hashdef{bwd1v1v1p1v1x0p1x0p0x1v1x0x0p0x1x1duals1v1v3x2x1}{17D1B13E1B960EFD}%
\hashdef{bwd1v1v1p1v1x1v1v1duals1v1v1v1}{A207DA908F977FE3}%
\hashdef{bwd1v1v1v1p1p1v1x0x0duals1v1v1x3x2}{859E46C1C5FE8578}%
\hashdef{bwd1v1v1v1p1v1x0p0x1v1x0p0x1duals1v1v1x2v2x1}{25781EDEA61320E6}%
\hashdef{bwd1v1v1v1p1v1x0p1x0duals1v1v1x2}{E26689FBED470E90}%
\hashdef{bwd1v1v1v1p1p1v1x0x0v1duals1v1v1x2x3v1}{C4B6B3927CB75C0E}%
\hashdef{gbg1v1p1v1x0p0x1}{14CE5A729ED34B23}%
\hashdef{gbg1v1p1v1x0p0x1p0x1}{A18A8E316A47B66A}%
\hashdef{gbg1v1p1v1x0p1x0}{52EBF913415F8806}%
\hashdef{gbg1v1p1v1x0p1x0p0x1}{11F859F1EB8C996C}%
\hashdef{gbg1v1p1v1x0p1x0p1x0}{A9DCCA734D7737B1}%

\newcommand{\bigraph}[1]{{\hspace{-3pt}\begin{array}{c}%
  \raisebox{-2.5pt}{\includegraphics[height=6mm]{\pathtographs \hashlookup{#1}}}%
\end{array}\hspace{-3pt}}}

\begin{document}
\title{Chirality and principal graph obstructions}
\author{David Penneys }
\date{\today}
\maketitle
\begin{abstract}
Determining which bipartite graphs can be principal graphs of subfactors is an important and difficult question in subfactor theory.
Using only planar algebra techniques, we prove a triple point obstruction which generalizes all known initial triple point obstructions to possible principal graphs.
We also prove a similar quadruple point obstruction with the same technique.
Using our obstructions, we eliminate some infinite families of possible principal graphs with initial triple and quadruple points which were a major hurdle in extending subfactor classification results  above index 5.
\end{abstract}

\section{Introduction}

Subfactor theory has many examples of unexpected discrete classification, beginning with Jones' index rigidity theorem \cite{MR696688}, which shows the index $[B\colon A]$ of a subfactor $A\subset B$ lies in the set $\{4\cos^2(\pi/k)|k\geq 3\}\cup[4,\infty]$.

We study a finite index subfactor by analyzing its standard invariant, which has many equivalent characterizations, such as Ocneanu's paragroups \cite{MR996454}, Popa's $\lambda$-lattices \cite{MR1334479}, and Jones' planar algebras \cite{math/9909027}.
The standard invariant also encodes the index and the principal graphs, which are the bipartite induction/restriction multi-graphs associated to tensoring with $\sb{A}B_B$ and $\sb{B}B_A$.

A second example of discrete classification is the $ADE$ classification of hyperfinite subfactors with index at most $4$, where $D_{\text{odd}}$ and $E_7$ do not occur \cite{MR996454,MR1278111}. 
Haagerup classified principal graphs of subfactors with index in the range $(4,3+\sqrt{3})$ \cite{MR1317352}, and recently, subfactor planar algebras with index less than 5 were completely classified \cite{MR2914056,MR2902285,MR2993924,MR2902286}. 
The recent survey \cite{bulletin} provides an excellent introduction to the subfactor classification program, along with state-of-the-art information on its progress.

These classifications have two main parts: restricting the list of possible principal graphs, and constructing examples when the graphs survive. 
The former task relies on \underline{principal graph obstructions}, which rule out many possible principal graphs by combinatorial constraints, or by relating local structure of the principal graph to data intrinsic to the subfactor planar algebra.

Several examples of the latter type of obstruction for triple points include Ocneanu's triple point obstruction \cite{MR1317352}, Jones' quadratic tangles obstruction \cite{MR2972458}, the triple-single obstruction \cite{MR2902285}, and Snyder's singly valent obstruction \cite{1207.5090}, which mutually generalizes the quadratic tangles and triple-single obstructions.

In this article, we prove a triple point obstruction which is strictly stronger than all known triple point obstructions for initial triple points, and a quadruple point obstruction of similar flavor.
The statement of the main theorem for initial triple points uses the following notation. 
Suppose that the principal graphs $(\Gamma_+,\Gamma_-)$ of a subfactor planar algebra have an initial triple point at depth $n-1$ (where $n\geq 2$), and that the projections one past the branch at depth $n$ of $\Gamma_+$ are labelled $P,Q$.
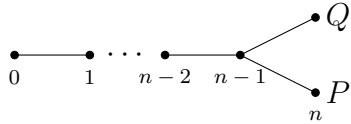
$$
\begin{tikzpicture}[baseline=-.1cm]
	\draw[fill] (-2,0) circle (0.05);
	\node at (-2,-.3) {\scriptsize{$0$}};	
	\draw (-2.,0.) -- (-1.,0.);
	\draw[fill] (-1,0) circle (0.05);
	\node at (-1,-.3) {\scriptsize{$1$}};
	\node at (-.5,0) {$\cdots$};
	\draw[fill] (0,0) circle (0.05);
	\node at (0,-.3) {\scriptsize{$n-2$}};
	\draw (0.,0.) -- (1.,0.);
	\draw[fill] (1.,0.) circle (0.05);
	\node at (1,-.3) {\scriptsize{$n-1$}};
	\draw (1.,0.) -- (2.,-0.5);
	\draw (1.,0.) -- (2.,0.5);
	\draw[fill] (2.,-0.5) circle (0.05);
	\node at (2.3,.5) {$Q$};
	\draw[fill] (2.,0.5) circle (0.05);
	\node at (2.3,-.5) {$P$};
	\node at (2,-.8) {\scriptsize{$n$}};	
\end{tikzpicture}
$$
In this case, $A=\frac{\Tr(P)}{\Tr(Q)}P-Q$ is a rotational eigenvector orthogonal to $\TL_{n,+}$ with rotational eigenvalue $\omega_A=\sigma_A^2$ (see Subsection \ref{sec:Annular} and the beginning of Section \ref{sec:TriplePoints}). We refer to $\omega_A$ (and sometimes $\sigma_A$) as the \underline{chirality} of the subfactor planar algebra.

\begin{thm*}
Suppose that for each $R$ at depth $n+1$ connected to $P$, there is a unique vertex $E(\overline{R})$ at depth $n$ connected to the dual vertex $\overline{R}$ of $R$.
Then there is an explicit formula for $\sigma_A+\sigma_A^{-1}$ in terms of the traces of the projections of $\Gamma_\pm$ with depth at most $n+1$. 
(See Theorem \ref{thm:TriplePoint} and Remark \ref{rem:Calculate} for more details.)
\end{thm*}

The importance of this formula is that it gives us the chirality, which is a priori hidden in the planar algebra structure, in terms of visible combinatorial data of the principal graph.
As corollaries, we obtain the obstructions of Jones and Snyder (Corollary \ref{cor:SinglyValent}) and a stronger version of Ocneanu's obstruction for initial triple points.

\begin{cor*}
Under the hypotheses of Ocneanu's obstruction for initial triple points, $\sigma_A+\sigma_A^{-1}=[n+2]-[n]$. (See Theorem \ref{thm:Ocneanu} for more details.)
\end{cor*}

After obtaining the chirality using the above corollary, we provide a quick proof that $D_{\text{odd}}$ and $E_7$ are not principal graphs by showing that the chirality is incompatible with the supertransitivity (see Remark \ref{rem:IndexAtMost4} and Examples \ref{exs:IndexAtMost4}). 

In Theorem \ref{thm:MagicNumbers11}, we obtain an obstruction for annular multiplicities $*11$ principal graphs, which were a major hurdle in extending subfactor classification results above index 5 (see Subsection \ref{sec:Annular} for the definition of annular multiplicities). 
In particular, we prove the following theorem, which eliminates a particular example arising from running the principal graph odometer \cite{MR2914056} above index 5.

\begin{thm*}
There is no subfactor with principal graphs a translated extension of 
$$
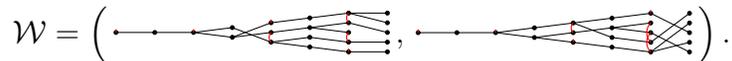
\cW=\left(\bigraph{bwd1v1v1p1v0x1p1x0p1x0v1x0x0p1x0x0p0x1x0p0x0x1v1x0x0x0p0x1x0x0p0x0x1x0p0x0x1x0p0x0x0x1v1x0x0x0x0p0x1x0x0x0p0x0x0x1x0p0x0x0x0x1p0x0x0x0x1duals1v1v2x1x3v1x3x2x5x4}, \bigraph{bwd1v1v1p1v1x0p1x0p0x1v1x0x0p0x0x1p0x1x0p0x0x1v1x0x0x0p0x1x0x0p0x0x1x0p0x0x1x0p0x0x0x1v0x0x1x0x0p1x0x0x0x0p0x0x0x0x1p1x0x0x0x0p0x1x0x0x0duals1v1v1x3x2v3x4x1x2x5}\right).
$$
(See Definition \ref{defn:translation} for the definition of ``translated extension," and see Subsection \ref{sec:11Weeds} for more details on eliminating $\cW$.)
\end{thm*}

Our techniques also apply to initial quadruple points which are even translated extensions of 
$$
\cQ=\left(\bigraph{bwd1v1p1p1duals1v1x3x2},\bigraph{bwd1v1p1p1duals1v1x3x2}\right).
$$
In this case, there are two rotational eigenvectors $A,B$ orthogonal to Temperley-Lieb one past the branch, and the chiralities $\omega_A=\sigma_A^2$ and $\omega_B=\sigma_B^2$ are distinct (see Proposition \ref{prop:LowWeight}). 
We get three equations in Theorem \ref{thm:QuadruplePoint}: two for the chirality $\sigma_A$, and one involving both $\sigma_A$ and $\sigma_B$. 
In a specific case, we obtain a quadruple point obstruction similar to Ocneanu's triple point obstruction in Theorem \ref{thm:OcneanuQuadruple}. 
We use this obstruction to eliminate two weeds in Corollary \ref{cor:EliminateQ}.

This article only relies on planar algebra techniques, and the one click rotation 
$$
\cF=
\begin{tikzpicture}[baseline=-.1cm]
	\draw (0,-.8)--(0,.8);
	\node at (0,1) {\scriptsize{$n-1$}};
	\node at (0,-1) {\scriptsize{$n-1$}};
	\draw (.2,.4) arc (180:0:.2cm) -- (.6,-1);
	\draw (-.2,-.4) arc (0:-180:.2cm) -- (-.6,1);
	\nbox{unshaded}{(0,0)}{0}{0}{}
\end{tikzpicture}
$$ 
plays a crucial role.
In particular, we use a quadratic relation due to Liu (Lemma \ref{lem:Liu}) which is a clever variant of Wenzl's relation \cite{MR873400}.
Such relations restrict the structure of the principal graph via the skein theory of the planar algebra, giving strong consequences. 
(For another such example of this phenomenon, see \cite{1208.1564}.)
Liu's relation is effective because it only involves four terms, two of which are in Temperley-Lieb and can be ignored for our purposes, and the one click rotation only appears on one side of the relation.

As the hypotheses of the main theorem above are quite general for initial triple and quadruple points, we expect that the obstructions obtained in this article will be very powerful in the classification of subfactors.

\subsection{Acknowledgements}

The author would like to thank Zhengwei Liu, Scott Morrison, and Noah Snyder for helpful conversations. 
Most of this work was conceived and completed during a workshop on fusion categories held at the Institut de Math\'{e}matiques de Bourgogne. The author would like to thank the organizers Peter Schauenburg and Siu-Hung Ng for their hospitality. 
The author was supported in part by the Natural Sciences and Engineering Research Council of Canada and DOD-DARPA grant HR0011-12-1-0009.

\section{Background}\label{sec:Background}

We refer the reader to \cite{MR2972458,MR2979509} for the definition of a subfactor planar algebra. 

\subsection{Notation for planar algebras}\label{sec:Notation}

In this article, $\cP_\bullet$ denotes a subfactor planar algebra of modulus $[2]$ where 
$$
[k]=\frac{q^k-q^{-k}}{q-q^{-1}},
$$ 
and $q\in \set{\exp\left(\frac{2\pi i}{2j}\right)}{j\geq 3}\cup[1,\I)$ such that $[2]=q+q^{-1}$.
We denote the Temperley-Lieb subfactor planar subalgebra by $\TL_\bullet$.

When we draw planar diagrams, we often suppress the external boundary disk. 
In this case, the external boundary is assumed to be a large rectangle whose distinguished interval contains the upper left corner. 
We draw one string with a number next to it instead of drawing that number of parallel strings. 
Since the shading can be inferred from the distinguished interval, the number of strands, and whether an element is in $\cP_{k,\pm}$, we omit it completely.
Finally, we usually draw elements of our planar algebras as rectangles with the same number of strands emanating from the top and bottom, and the distinguished interval is always on the left unless it is marked otherwise.

We refer the reader to \cite{MR999799,1208.1564,bulletin} for the definition of the principal graphs $(\Gamma_+,\Gamma_-)$ of $\cP_\bullet$.
If there is only one projection $P$ in the equivalence class $[P]$ corresponding to a vertex of $\Gamma_\pm$, then we identify $[P]$ with $P$.

For a projection $P\in\cP_{k,\pm}$, the dual projection $\overline{P}$ is given by
$$
\overline{P}
=
\begin{tikzpicture}[baseline=-.1cm]
	\clip (-1.1,-.9)--(-1.1,.9)--(1.1,.9)--(1.1,-.9);	
	\draw (0,.4) arc (180:0:.4cm)--(.8,-.8);
	\draw (0,-.4) arc (0:-180:.4cm)--(-.8,.8);
	\draw[thick, unshaded] (-.4, -.4) -- (-.4, .4) -- (.4, .4) -- (.4,-.4) -- (-.4, -.4);
	\node at (0,0) {$P$};
	\node at (-1,.6) {{\scriptsize{$k$}}};
	\node at (1,-.6) {{\scriptsize{$k$}}};
\end{tikzpicture}.
$$
For a vertex $[P]$ of $\Gamma_\pm$ at depth $n$, there is a corresponding dual vertex $[\overline{P}]$ necessarily at depth $n$.
If $n$ is even then $[\overline{P}]$ is a vertex of $\Gamma_+$, but if $n$ is odd, then $[\overline{P}]$ is a vertex of $\Gamma_-$.

When we draw principal graph pairs, we use the convention that the vertical ordering of vertices at a given odd depth determines the duality; the lowest vertices in each graph at each odd depth are dual to each other, etc.
When we specify the duality at even depths (sometimes we omit this data), the duality is represented by red arcs joining dual pairs of vertices. 
Self-dual even vertices have a small red dash above them.

\begin{ex}
The Haagerup subfactor \cite{MR1317352,MR1686551} has principal graphs 
$$
\left(\bigraph{bwd1v1v1v1p1v1x0p0x1v1x0p0x1duals1v1v1x2v2x1}, \bigraph{bwd1v1v1v1p1v1x0p1x0duals1v1v1x2}\right).
$$
\end{ex}

\subsection{Supertransitivity and annular multiplicities}\label{sec:Annular}

We rapidly recall the notions of supertransitivity and annular multiplicities for subfactor planar algebras and potential principal graphs following \cite{MR1929335,MR2972458,MR2902285,bulletin}.

\begin{defn}
A subfactor planar algebra $\cP_\bullet$ is called \underline{$k$-supertransitive} if $\cP_{j,+}=\TL_{j,+}$ for all $0\leq j\leq k$. Equivalently, $\cP_\bullet$ is $k$-supertransitive if the truncation $\Gamma_+(k)$ of $\Gamma_+$ to depth $k$ is $A_{k+1}$.
This gives us the notion of the supertransitivity of a potential principal graph.
\end{defn}

\begin{rem}
When we say a subfactor planar algebra or potential principal graph $\Gamma_\pm$ is $k$-supertransitive, we usually also mean that $\TL_{k+1,+}\subsetneq \cP_{k+1,+}$ or $\Gamma_\pm(k+1)\neq A_{k+2}$, although strictly speaking, this is an abuse of nomenclature.
\end{rem}

Every subfactor planar algebra decomposes into an orthogonal direct sum of irreducible annular Temperley-Lieb submodules.
Each irreducible representation that appears in the direct sum is generated by a single low weight rotational eigenvector at some depth $n$, for which the rotational eigenvalue is an $n$-th root of unity \cite{MR1929335}.
(Other eigenvalues are possible for $0$-boxes which do not occur in subfactor planar algebras.)
Hence $\cP_{n,\pm}$ is the direct sum of the annular consequences $AC_{n,\pm}$ of $\cP_{n-1,\pm}$ and the new low weight vectors at depth $n$.

\begin{defn}
The sequence of \underline{annular multiplicities} of a subfactor planar algebra is the sequence of multiplicities of lowest weight 
vectors, ignoring eigenvalues, i.e.,
$$
a_n = \dim(\cP_{n,+})-\dim(AC_{n,+}).
$$
If $\cP_\bullet$ is $n-1$ supertransitive, then $a_n$ is just called the \underline{multiplicity}.
\end{defn}

\begin{rem}
A formula for the annular multiplicities in terms of the principal graph is given in \cite{MR1929335}. We give it here for the reader's convenience. First, $\dim(\cP_{j,\pm})=L_{j}$, the number of loops of length $2j$ on $\Gamma_\pm$ starting at $\star$. We then have
$$
a_k = \sum_{j=0}^k (-1)^{j-k}\frac{2k}{k+j}{k+j \choose k-j}L_j.
$$
Thus it makes sense to discuss the annular multiplicities of a subfactor planar algebra or of a candidate principal graph pair.
\end{rem}

The zeroth annular multiplicity $a_0$ of a subfactor planar algebra is always 1, as the empty diagram generates $\TL_\bullet$ as an annular Temperley-Lieb module. 
If $\cP_\bullet$ is $k$-supertransitive, then for all $1\leq j\leq k$, $a_j=0$, since $\cP_{j,\pm}=\TL_{j,\pm}$. 
For simplicity of notation, if $\cP_\bullet$ is $k$-supertransitive, but not $k+1$-supertransitive, we denote the initial chain $(a_0=1,a_1=0,\dots,a_{k}=0)$ of annular multiplicities by an asterisk, and we begin the sequence after the asterisk with $a_{k+1}$.  

\begin{defn}\label{defn:translation}
We say $\Gamma_\pm$ is a \underline{translation} of $\Gamma_\pm^0$ if $\Gamma_\pm$ is obtained from $\Gamma_\pm^0$ by increasing the supertransitivity. (If the even dual data of $\Gamma_\pm^0$ is specified, we only allow even translations.)
We say $\Gamma_\pm$ is an \underline{extension} of $\Gamma_\pm^0$ if $\Gamma_\pm$ is obtained from $\Gamma_\pm^0$ by adding new vertices and edges at strictly greater depths than the maximum depth of any vertex in $\Gamma_\pm^0$.

A \underline{weed} is a graph pair $\cW$ which represents an infinite family of possible principal graphs obtained from $\cW$ by translation and extension.
A \underline{vine} is a graph pair $\cV$ which represents an infinite family of possible principal graphs obtained from $\cV$ by translation only.
\end{defn}

\begin{exs}
\mbox{}
\begin{enumerate}[(1)]
\item
If $\cP_\bullet$ has annular multiplicities $*10$, then $\Gamma_+$ is a translated extension of
$$
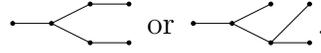
\bigraph{gbg1v1p1v1x0p0x1}
\text{ or }
\bigraph{gbg1v1p1v1x0p1x0}.
$$
\item
If $\cP_\bullet$ has annular multiplicities $*11$, then $\Gamma_+$ is a translated extension of
$$
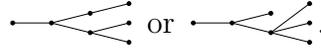
\bigraph{gbg1v1p1v1x0p1x0p0x1}
\text{ or }
\bigraph{gbg1v1p1v1x0p1x0p1x0}.
$$
\end{enumerate}
\end{exs}

Suppose $\cP_\bullet$ is $n-1$ supertransitive with multiplicity $m$. Then there are $m$ new low weight rotational eigenvectors $A_1,\dots, A_m\in \cP_{n,+}$, and $\cP_{n,+}=\TL_{n,+}\oplus \spann\{A_1,\dots, A_m\}$.

\begin{defn}
The \underline{chiralities} of $\cP_\bullet$ are the rotational eigenvalues $\omega_{A_1},\dots, \omega_{A_m}$ corresponding to the new low weight rotational eigenvectors $A_1,\dots, A_m$ respectively.
\end{defn}

\begin{rems}
\mbox{}
\begin{enumerate}[(1)]
\item
The chirality of a subfactor planar algebra was first defined for some special cases in \cite[Definition 4.2.12]{math/9909027} and subsequently studied in \cite{MR2972458}.
\item
In this article, work with a (completely determined{!}) square root $\sigma_{A}$ of the chirality $\omega_{A}$, and by a slight abuse of notation, we also call $\sigma_A$ a chirality of our subfactor planar algebra. 
This abuse of notation was also used in \cite{1208.3637}.
\end{enumerate}
\end{rems}

\subsection{Liu's relation}

In this subsection, we derive a strong quadratic relation due to Liu, which is a clever variant of Wenzl's relation. 
Let $\cP_\bullet$ be a subfactor planar algebra with principal graphs $(\Gamma_+,\Gamma_-)$.
We begin by stating Wenzl's relation.
We will then state a generalization from which Liu's relation quickly follows. 
\begin{lem}[Wenzl's relation \cite{MR873400}]\label{lem:Wenzl}
Let $\jw{k}\in \TL_{k,\pm}$ be the $k$-th Jones-Wenzl projection.
Then 
$$
\begin{tikzpicture}[baseline = -.1cm]
	\clip (-.8,-.7)--(-.8,.7)--(.8,.7)--(.8,-.7);
	\draw (0,.8)--(0,-1.8);
	\node at (-.4,.6) {{\scriptsize{$k+1$}}};
	\node at (-.4,-.6) {{\scriptsize{$k+1$}}};
	\filldraw[unshaded,thick] (-.6,-.4) rectangle (.6,.4);
	\node at (0,0) {$\jw{k+1}$};
\end{tikzpicture}
=
\begin{tikzpicture}[baseline = -.1cm]
	\clip (-.8,-.8)--(-.8,.8)--(.7,.8)--(.7,-.8);
	\draw (0,.8)--(0,-1.8);
	\node at (-.2,.6) {{\scriptsize{$k$}}};
	\node at (-.2,-.6) {{\scriptsize{$k$}}};
	\filldraw[unshaded,thick] (-.4,.4)--(.4,.4)--(.4,-.4)--(-.4,-.4)--(-.4,.4);
	\node at (0,0) {$\jw{k}$};
	\draw (.6,.8)--(.6,-.8);
\end{tikzpicture}
-
\D\frac{[k]}{[k+1]}
\begin{tikzpicture}[baseline = -.7cm]
	\clip (-.8,-1.9)--(-.8,.9)--(.8,.9)--(.8,-1.9);
	\draw (0,.9)--(0,-1.9);
	\node at (-.2,.8) {{\scriptsize{$k$}}};
	\node at (-.2,-1.8) {{\scriptsize{$k$}}};
	\node at (-.4,-.5) {{\scriptsize{$k-1$}}};
	\filldraw[unshaded,thick] (-.4,.6)--(.4,.6)--(.4,-.2)--(-.4,-.2)--(-.4,.6);
	\node at (0,.2) {$\jw{k}$};
	\filldraw[unshaded,thick] (-.4,-.8)--(.4,-.8)--(.4,-1.6)--(-.4,-1.6)--(-.4,-.8);
	\node at (0,-1.2) {$\jw{k}$};
	\draw (.2,-.2) arc (180:360:.2cm) -- (.6,1);
	\draw (.2,-.8) arc (180:0:.2cm) -- (.6,-1.8);
\end{tikzpicture}\,.
$$
\end{lem}
Below is a more general version of Wenzl's relation known to experts (e.g., see \cite[Lemma 3.7]{MR2922607}).

\begin{lem}[Generalized Wenzl's relation]\label{lem:GeneralizedWenzl}
Suppose that $P\in\cP_{k,\pm}$ is a projection corresponding to the vertex $[P]$ at depth $k$ of $\Gamma_\pm$.
Suppose that $f\in\cP_{k-1,\pm}$ is a projection corresponding to the vertex $[f]$ at depth $k-1$ of $\Gamma_\pm$.
Suppose $[P]$ and $[f]$ are connected by a single edge, and $[f]$ is the only vertex at depth $k-1$ attached to $[P]$.
Then the unique projection isomorphic to $f$ under the image of $P$ in $\cP_{k+1,\pm}$ is given by
$$
\D\frac{\Tr(f)}{\Tr(P)}
\begin{tikzpicture}[baseline = -.7cm]
	\clip (-.8,-1.9)--(-.8,.9)--(.8,.9)--(.8,-1.9);
	\draw (0,.9)--(0,-1.9);
	\node at (-.2,.8) {{\scriptsize{$k$}}};
	\node at (-.2,-1.8) {{\scriptsize{$k$}}};
	\node at (-.4,-.5) {{\scriptsize{$k-1$}}};
	\filldraw[unshaded,thick] (-.4,.6)--(.4,.6)--(.4,-.2)--(-.4,-.2)--(-.4,.6);
	\node at (0,.2) {$P$};
	\filldraw[unshaded,thick] (-.4,-.8)--(.4,-.8)--(.4,-1.6)--(-.4,-1.6)--(-.4,-.8);
	\node at (0,-1.2) {$P$};
	\draw (.2,-.2) arc (180:360:.2cm) -- (.6,1);
	\draw (.2,-.8) arc (180:0:.2cm) -- (.6,-1.8);
\end{tikzpicture}.
$$
Hence the  ``new stuff" \cite{MR999799} $P'$ connected to $P$ at depth $k+1$ is given by
\begin{equation*}
\begin{tikzpicture}[baseline = -.1cm]
	\clip (-.8,-.7)--(-.8,.7)--(.8,.7)--(.8,-.7);
	\draw (0,.8)--(0,-1.8);
	\node at (-.4,.6) {{\scriptsize{$k+1$}}};
	\node at (-.4,-.6) {{\scriptsize{$k+1$}}};
	\filldraw[unshaded,thick] (-.5,.4)--(.5,.4)--(.5,-.4)--(-.5,-.4)--(-.5,.4);
	\node at (0,0) {$P'$};
\end{tikzpicture}
=
\begin{tikzpicture}[baseline = -.1cm]
	\clip (-.8,-.8)--(-.8,.8)--(.7,.8)--(.7,-.8);
	\draw (0,.8)--(0,-1.8);
	\node at (-.2,.6) {{\scriptsize{$k$}}};
	\node at (-.2,-.6) {{\scriptsize{$k$}}};
	\filldraw[unshaded,thick] (-.4,.4)--(.4,.4)--(.4,-.4)--(-.4,-.4)--(-.4,.4);
	\node at (0,0) {$P$};
	\draw (.6,.8)--(.6,-.8);
\end{tikzpicture}
-
\D\frac{\Tr(f)}{\Tr(P)}
\begin{tikzpicture}[baseline = -.7cm]
	\clip (-.8,-1.9)--(-.8,.9)--(.8,.9)--(.8,-1.9);
	\draw (0,.9)--(0,-1.9);
	\node at (-.2,.8) {{\scriptsize{$k$}}};
	\node at (-.2,-1.8) {{\scriptsize{$k$}}};
	\node at (-.4,-.5) {{\scriptsize{$k-1$}}};
	\filldraw[unshaded,thick] (-.4,.6)--(.4,.6)--(.4,-.2)--(-.4,-.2)--(-.4,.6);
	\node at (0,.2) {$P$};
	\filldraw[unshaded,thick] (-.4,-.8)--(.4,-.8)--(.4,-1.6)--(-.4,-1.6)--(-.4,-.8);
	\node at (0,-1.2) {$P$};
	\draw (.2,-.2) arc (180:360:.2cm) -- (.6,1);
	\draw (.2,-.8) arc (180:0:.2cm) -- (.6,-1.8);
\end{tikzpicture}.
\end{equation*}
\end{lem}

We now prove Liu's relation from this generalization of Wenzl's relation.
Suppose that $\cP_\bullet$ is $(n-1)$ supertransitive for some $n\geq 2$, such that $\Gamma_+$ is simply laced when truncated to depth $n$, i.e., $\cP_{n,+}\ominus \TL_{n,+}$ is abelian.

\begin{lem}[Liu's relation]\label{lem:Liu}
Let $P$ be a projection at depth $n$ of $\Gamma_+$, and let $P'$ be as in Lemma \ref{lem:GeneralizedWenzl}.
Then
\begin{align*}
\cF(\overline{P})^*\cF(\overline{P})
=
\begin{tikzpicture}[baseline=-.1cm]
	\draw (0,-1.4)--(0,1.4);
	\node at (0,1.6) {\scriptsize{$n-1$}};
	\node at (0,-1.6) {\scriptsize{$n-1$}};
	\draw (.2,1.1) arc (180:0:.2cm) -- (.6,-1.1) arc (0:-180:.2cm);
	\draw (-.2,.3) arc (0:-180:.2cm) -- (-.6,1.6);
	\draw (-.2,-.3) arc (0:180:.2cm) -- (-.6,-1.6);
	\nbox{unshaded}{(0,-.7)}{0}{0}{\overline{P}}
	\nbox{unshaded}{(0,.7)}{0}{0}{\overline{P}}
\end{tikzpicture}
&=
\frac{\Tr(P)^2}{[n]^2}
\left(
\begin{tikzpicture}[baseline=-.1cm]
	\draw (0,-.8)--(0,.8);
	\node at (-.2,.6) {\scriptsize{$n$}};
	\node at (-.2,-.6) {\scriptsize{$n$}};
	\nbox{unshaded}{(0,0)}{0}{0}{\jw{n}}
\end{tikzpicture}
+
\frac{[n-1]}{[n]}
\begin{tikzpicture}[baseline=-.1cm]
	\draw (.2,-1.4)--(.2,1.4);
	\node at (.2,1.6) {\scriptsize{$n-1$}};
	\node at (.2,-1.6) {\scriptsize{$n-1$}};
	\draw (-.2,.3) arc (0:-180:.2cm) -- (-.6,1.6);
	\draw (-.2,-.3) arc (0:180:.2cm) -- (-.6,-1.6);
	\nbox{unshaded}{(.2,-.7)}{.2}{.2}{\jw{n-1}}
	\nbox{unshaded}{(.2,.7)}{.2}{.2}{\jw{n-1}}
\end{tikzpicture}
\right)
-
\frac{\Tr(P)}{[n]}\,
\begin{tikzpicture}[baseline=-.1cm]
	\draw (0,-.8)--(0,.8);
	\node at (0,1) {\scriptsize{$n$}};
	\node at (0,-1) {\scriptsize{$n$}};
	\draw (.2,.4) arc (180:0:.2cm) -- (.6,-.4) arc (0:-180:.2cm);
	\nbox{unshaded}{(0,0)}{0}{0}{\overline{P'}}
\end{tikzpicture}
\,.
\end{align*}
\end{lem}
\begin{proof}
Start with $P$, apply the generalization of Wenzl's relation, rotate by 180$^\circ$, and then apply a cap on the right to see that
\begin{align*}
\begin{tikzpicture}[baseline=-.1cm]
	\draw (0,-1.4)--(0,1.4);
	\node at (0,1.6) {\scriptsize{$n-1$}};
	\node at (0,-1.6) {\scriptsize{$n-1$}};
	\draw (.2,1.1) arc (180:0:.2cm) -- (.6,-1.1) arc (0:-180:.2cm);
	\draw (-.2,.3) arc (0:-180:.2cm) -- (-.6,1.6);
	\draw (-.2,-.3) arc (0:180:.2cm) -- (-.6,-1.6);
	\nbox{unshaded}{(0,-.7)}{0}{0}{\overline{P}}
	\nbox{unshaded}{(0,.7)}{0}{0}{\overline{P}}
\end{tikzpicture}
&=
\frac{\Tr(P)}{[n]}
\left(
\begin{tikzpicture}[baseline=-.1cm]
	\draw (0,-.8)--(0,.8);
	\node at (0,1) {\scriptsize{$n-1$}};
	\node at (0,-1) {\scriptsize{$n-1$}};
	\draw (.2,.4) arc (180:0:.2cm) -- (.6,-.4) arc (0:-180:.2cm);
	\draw (-.7,-.8)--(-.7,.8);
	\nbox{unshaded}{(0,0)}{0}{0}{\overline{P}}
\end{tikzpicture}
-
\begin{tikzpicture}[baseline=-.1cm]
	\draw (0,-.8)--(0,.8);
	\node at (0,1) {\scriptsize{$n$}};
	\node at (0,-1) {\scriptsize{$n$}};
	\draw (.2,.4) arc (180:0:.2cm) -- (.6,-.4) arc (0:-180:.2cm);
	\nbox{unshaded}{(0,0)}{0}{0}{\overline{P'}}
\end{tikzpicture}
\right)
.
\end{align*}
Liu's observation is that the left hand side is equal to $\cF(\overline{P})^*\cF(\overline{P})$, where 
$$
\cF=
\begin{tikzpicture}[baseline=-.1cm]
	\draw (0,-.8)--(0,.8);
	\node at (0,1) {\scriptsize{$n-1$}};
	\node at (0,-1) {\scriptsize{$n-1$}};
	\draw (.2,.4) arc (180:0:.2cm) -- (.6,-1);
	\draw (-.2,-.4) arc (0:-180:.2cm) -- (-.6,1);
	\nbox{unshaded}{(0,0)}{0}{0}{}
\end{tikzpicture}
$$ 
is the one click rotation.
The rest is a straightforward calculation, since capping $\overline{P}$ gives a multiple of $\jw{n-1}$ (taking traces gives the multiple).
Note that we use Wenzl's relation again for $\jw{n-1}$ with a strand on the left to get the final formula.
\end{proof}

\begin{rem}
When the truncation $\Gamma_\pm(n+1)$ of $\Gamma_\pm$ to depth $n+1$ is simply laced and acyclic, which ensures that we can determine the rightmost term in Liu's relation (see Remark \ref{rem:Calculate}), we can use Liu's relation to determine the modulus of the entries of the one click rotation at depth $n$.
\end{rem}

\section{Multiplicity 1 - triple points}\label{sec:TriplePoints}

When $\Gamma_+$ is $(n-1)$ supertransitive for some $n\geq 2$ and has multiplicity 1, we can identify the rotational low weight eigenvector. 
We use the following conventions:
\begin{itemize}
\item designate the projections $P,Q$ and $\check{P},\check{Q}$ at depth $n$ of $\Gamma_+$ and $\Gamma_-$ respectively,
\item if $n$ is odd, then choose so that $\overline{P}=\check{P}$ and $\overline{Q}=\check{Q}$,
\item $\D r= \frac{\Tr(Q)}{\Tr(P)}$ is the \underline{branch factor},
\item $A=rP-Q\in \cP_{n,+}$ is a low-weight rotational eigenvector with eigenvalue $\omega_A$,
\item since $P+Q=\jw{n}$, we have $\D P=\frac{\jw{n}+A}{1+r}$ and $\D Q = \frac{r\jw{n}-A}{1+r}$,
\item $A^2=(r-1)A+r\jw{n}$, 
\item the previous 4 lines have analogous formulas with checks, and
\item $\D\cF(A)=\frac{\sqrt{r}}{\sqrt{\check{r}}} \sigma_A  \check{A}$ for a determined $\sigma_A^2=\omega_A$.
\end{itemize}

\begin{note}
By designating $P,\check{P}$, we have designated $A,\check{A}$, which completely determines $\sigma_A$.
\end{note}

\begin{defn}\label{defn:PPrime}
Since $P$ only connects to $\jw{n-1}$ at depth $n-1$, we may define $P'$ as in Lemma \ref{lem:GeneralizedWenzl}. 
Let $\cap_{n+1}(\overline{P'})$ denote $\overline{P'}$ capped off on the right as in Lemma \ref{lem:Liu}:
$$
\cap_{n+1}(\overline{P'})=\,
\begin{tikzpicture}[baseline=-.1cm]
	\draw (0,-.8)--(0,.8);
	\node at (0,1) {\scriptsize{$n$}};
	\node at (0,-1) {\scriptsize{$n$}};
	\draw (.2,.4) arc (180:0:.2cm) -- (.6,-.4) arc (0:-180:.2cm);
	\nbox{unshaded}{(0,0)}{0}{0}{\overline{P'}}
\end{tikzpicture}\,.
$$
Since capping $\cap_{n+1}(\overline{P'})$ on the top or bottom gives zero, we must have that $\cap_{n+1}(\overline{P'})\in \spann\{A_0,\jw{n}\}\subset \cP_{n,\pm}$, where 
$$
A_0=
\begin{cases}
\check{A} &\text{if $n$ is even}\\
A &\text{if $n$ is odd.}
\end{cases}
$$
We denote the coefficient of $A_0$ in $\cap_{n+1}(\overline{P'})$ by
$
\underset{\in \cap_{n+1}(\overline{P'})}{\coeff}\left(A_0\right).
$
\end{defn}

\begin{thm}\label{thm:TriplePoint}
If $n$ is even and $\overline{P}=P$, then
\begin{equation}\label{eqn:EvenObstructionSelfDual}
(\check{r}-1)\frac{r}{\check{r}} 
- 
\frac{(\sigma_A+\sigma_A^{-1})}{[n]}\frac{\sqrt{r}}{\sqrt{\check{r}}}
=
-(1+r) \frac{[n+1]}{[n]} \left(\underset{\in \cap_{n+1}(\overline{P'})}{\coeff}\left(\check{A}\right)\right).
\tag{E}
\end{equation}
If $n$ is even and $\overline{P}=Q$, then $r=1$, and
\begin{equation}\label{eqn:EvenObstructionNotSelfDual}
(\check{r}-1)\frac{1}{\check{r}} 
+
\frac{(\sigma_A+\sigma_A^{-1})}{[n]\sqrt{\check{r}}}
=
-2 \frac{[n+1]}{[n]} \left(\underset{\in \cap_{n+1}(\overline{P'})}{\coeff}\left(\check{A}\right)\right).
\tag{$\overline{\text{E}}$}
\end{equation}
If $n$ is odd, then $r=\check{r}$, and
\begin{equation}\label{eqn:OddObstruction}
(r-1) 
-
\frac{(\sigma_A+\sigma_A^{-1})}{[n]}
=
-(1+r) \frac{[n+1]}{[n]} \left(\underset{\in \cap_{n+1}(\overline{P'})}{\coeff}\left(A\right)\right).
\tag{O}
\end{equation}
\end{thm}
\begin{proof}
To prove Equations \eqref{eqn:EvenObstructionSelfDual} and \eqref{eqn:EvenObstructionNotSelfDual}, we find the coefficient of $\check{A}$ in both sides of Liu's relation, noting that $\check{A}$ is orthogonal to $\TL_{n,-}$. 
Suppose $\overline{P}=P$. 
Then since $\D P=\frac{\jw{n}+ A}{1+r}$, the left hand side in Liu's relation is equal to
\begin{align*}
\cF(P)^*\cF(P)
&=
\frac{1}{(1+r)^2}\cF(\jw{n}+ A)^*\cF(\jw{n}+ A)
\\&=
\frac{1}{(1+r)^2} \left(
\frac{r}{\check{r}}\check{A}^2 +
\begin{tikzpicture}[baseline=-.1cm]
	\draw (0,-1.4)--(0,1.4);
	\node at (0,1.6) {\scriptsize{$n-1$}};
	\node at (0,-1.6) {\scriptsize{$n-1$}};
	\draw (.2,1.1) arc (180:0:.2cm) -- (.6,-1.1) arc (0:-180:.2cm);
	\draw (-.2,.3) arc (0:-180:.2cm) -- (-.6,1.6);
	\draw (-.2,-.3) arc (0:180:.2cm) -- (-.6,-1.6);
	\nbox{unshaded}{(0,-.7)}{0}{0}{\jw{n}}
	\nbox{unshaded}{(0,.7)}{0}{0}{A}
\end{tikzpicture}
+
\begin{tikzpicture}[baseline=-.1cm]
	\draw (0,-1.4)--(0,1.4);
	\node at (0,1.6) {\scriptsize{$n-1$}};
	\node at (0,-1.6) {\scriptsize{$n-1$}};
	\draw (.2,1.1) arc (180:0:.2cm) -- (.6,-1.1) arc (0:-180:.2cm);
	\draw (-.2,.3) arc (0:-180:.2cm) -- (-.6,1.6);
	\draw (-.2,-.3) arc (0:180:.2cm) -- (-.6,-1.6);
	\nbox{unshaded}{(0,-.7)}{0}{0}{A}
	\nbox{unshaded}{(0,.7)}{0}{0}{\jw{n}}
\end{tikzpicture}
+
\begin{tikzpicture}[baseline=-.1cm]
	\draw (0,-1.4)--(0,1.4);
	\node at (0,1.6) {\scriptsize{$n-1$}};
	\node at (0,-1.6) {\scriptsize{$n-1$}};
	\draw (.2,1.1) arc (180:0:.2cm) -- (.6,-1.1) arc (0:-180:.2cm);
	\draw (-.2,.3) arc (0:-180:.2cm) -- (-.6,1.6);
	\draw (-.2,-.3) arc (0:180:.2cm) -- (-.6,-1.6);
	\nbox{unshaded}{(0,-.7)}{0}{0}{\jw{n}}
	\nbox{unshaded}{(0,.7)}{0}{0}{\jw{n}}
\end{tikzpicture}
\right).
\end{align*}
The fourth diagram is in $\TL_{n,-}$ and does not contribute to the coefficient of $\check{A}$.
Note that the only diagram in the $\jw{n}$ which contributes to the second, respectively third, summand is 
$
\begin{tikzpicture}[baseline = -.1cm,xscale=-1]
	\draw[thick] (-.6,-.4)--(-.6,.4)--(.75,.4)--(.75,-.4)--(-.6,-.4);
	\draw (-.4,.4)--(.2,-.4);
	\draw (-.4,-.4) arc (180:0:.2cm);
	\draw (-.2,.4) arc (-180:0:.2cm);
	\node at (.35,0) {{\scriptsize{$n-2$}}};
\end{tikzpicture}
$\,,
respectively
$
\begin{tikzpicture}[baseline = -.1cm]
	\draw[thick] (-.6,-.4)--(-.6,.4)--(.75,.4)--(.75,-.4)--(-.6,-.4);
	\draw (-.4,.4)--(.2,-.4);
	\draw (-.4,-.4) arc (180:0:.2cm);
	\draw (-.2,.4) arc (-180:0:.2cm);
	\node at (.35,0) {{\scriptsize{$n-2$}}};
\end{tikzpicture}
$, both of which have coefficient $(-1)^{n-1}/[n]$ \cite{morrison,MR2375712}. Hence the coefficient of $\check{A}$ is
$$
\frac{1}{(1+r)^2} \left((\check{r}-1)\frac{r}{\check{r}}+(\sigma_A+\sigma_A^{-1})\frac{(-1)^{n-1}}{[n]}\frac{\sqrt{r}}{\sqrt{\check{r}}} \right).
$$
Since the first two diagrams on the right hand side of Liu's relation are in $\TL_{n,-}$, Equation \eqref{eqn:EvenObstructionSelfDual} follows immediately using $n$ is even and $\D \Tr(P)=\frac{[n+1]}{1+r}$.

When $\overline{P}=Q$, a similar argument shows the coefficient of $\check{A}$ in $\cF(Q)^*\cF(Q)$ is given by
$$
\frac{1}{(1+r)^2} \left((\check{r}-1)\frac{r}{\check{r}}-r(\sigma_A+\sigma_A^{-1})\frac{(-1)^{n-1}}{[n]}\frac{\sqrt{r}}{\sqrt{\check{r}}} \right).
$$ 
Now use $r=1$ and $n$ is even, and substitute for $\Tr(P)$.

For Equation \eqref{eqn:OddObstruction}, we find the coefficient of $A$ in $\cF(\overline{P})^*\cF(\overline{P})=\cF(\check{P})^*\cF(\check{P})$. 
A similar argument as before shows this coefficient is given by
$$
\frac{1}{(1+\check{r})^2} \left((r-1)\frac{\check{r}}{r}+(\sigma_A+\sigma_A^{-1})\frac{(-1)^{n-1}}{[n]}\frac{\sqrt{\check{r}}}{\sqrt{r}} \right).
$$
Now use $r=\check{r}$ and $n$ is odd, and substitute for $\Tr(P)$.
\end{proof}

\begin{rem}\label{rem:Calculate}
To get useful formulas from Theorem \ref{thm:TriplePoint}, we need to be able to determine $\cap_{n+1}(\overline{P'})$. 
This is always possible under the following condition:
\begin{itemize}
\item
For each $R$ at depth $n+1$ of $\Gamma_+$, let $m_R$ be the number of edges connecting $P$ and $R$.
For each $R$ with $m_R>0$, suppose there is a unique vertex $E(\overline{R})$ at depth $n$ connected to $\overline{R}$. In this case,
$$
\underset{\in \cap_{n+1}(\overline{P'})}{\coeff}\left(A_0\right)=\sum_R \frac{m_R\Tr(R)}{\Tr(E(\overline{R}))} \underset{\in E(\overline{R})}{\coeff}\left(A_0\right).
$$
\end{itemize}
This very general condition is satisfied by most principal graphs of small index, after choosing the appropriate $P$.
\end{rem}


\subsection{Triple point obstructions}

We now recover all known triple point obstructions from Theorem \ref{thm:TriplePoint}, and we give a new annular multiplicities $*11$ obstruction.

We start with Ocneanu's triple point obstruction for an initial triple point.
The hypotheses of Theorem \ref{thm:Ocneanu} are equivalent to those for Ocneanu's triple point obstruction \cite{MR1317352} (see \cite[Theorems 3.5 and 3.6]{MR2902285} for more details). Our proposition actually determines the chirality, which is stronger than Ocneanu's result, but only works for an initial triple point, which is weaker than Ocneanu's result.

\begin{thm}\label{thm:Ocneanu}
Suppose that
\begin{itemize}
\item 
if $n$ is even, then $r=\check{r}$, and for every vertex $R$ connected to  $P$, $\overline{R}$ only connects to $\check{P}$ if $\overline{P}=P$ or $\check{Q}$ if $\overline{P}=Q$, or
\item
if $n$ is odd, then for every vertex $R$ connected to  $P$, $\overline{R}$ only connects to $P$.
\end{itemize}
Then $\sigma_A+\sigma_A^{-1}=[n+2]-[n]=q^{n+1}+q^{-n-1}$.
\end{thm}
\begin{proof}
Both conditions for $n$ even and $n$ odd imply that $r=\check{r}$ and
$$
\Tr(P')=[2]\Tr(P)-[n]=[2]\left(\frac{[n+1]}{1+r}\right)-[n]=\frac{[n+2]-r[n]}{1+r}.
$$
Under these hypotheses, all three equations given in Theorem \ref{thm:TriplePoint} specialize to give $\sigma_A+\sigma_A^{-1}=[n+2]-[n]$.
We include the proof below for convenience.

For $n$ even and $\overline{P}=P$, $\D\cap_{n+1}(\overline{P'})=\frac{\Tr(P')}{\Tr(P)}\check{P}$, and Equation \eqref{eqn:EvenObstructionSelfDual} rearranges and simplifies to give
\begin{align*}
\sigma_A+\sigma_A^{-1}
&=
[n]\left((r-1)+(1+r)\frac{[n+1]}{[n]} \left(\underset{\in \cap_{n+1}(\overline{P'})}{\coeff}\left(\check{A}\right)\right)\right)
\\&=
(r-1)[n]+(1+r)[n+1]\left(\frac{\Tr(P')}{\Tr(P)}\cdot \frac{1}{1+r}\right)
\\&=
(r-1)[n]+(1+r)\left(\frac{[n+2]-r[n]}{1+r}\right)
\\&=[n+2]-[n].
\end{align*}

For $n$ even and $\overline{P}=Q$, $\Tr(P)=\Tr(Q)$, so $r=\check{r}=1$, $\D\cap_{n+1}(\overline{P'})=\frac{\Tr(P')}{\Tr(P)}\check{Q}$, and Equation \eqref{eqn:EvenObstructionNotSelfDual} rearranges and simplifies to give
\begin{align*}
\sigma_A+\sigma_A^{-1}
&=
-2[n+1]\left(\underset{\in \cap_{n+1}(\overline{P'})}{\coeff}\left(\check{A}\right)\right)
=
-2[n+1]\left(\frac{\Tr(P')}{\Tr(P)}\cdot \frac{-1}{1+r}\right)
=[n+2]-[n].
\end{align*}

Finally, For $n$ odd, $\D\cap_{n+1}(\overline{P'})=\frac{\Tr(P')}{\Tr(P)}P$, and Equation \eqref{eqn:OddObstruction} rearranges and simplifies exactly as Equation \eqref{eqn:EvenObstructionSelfDual} did for the case $n$ even and $\overline{P}=P$.
\end{proof}

\begin{cor}[Ocneanu \cite{MR1317352}]\label{cor:Ocneanu}
Under the conditions of Theorem \ref{thm:Ocneanu}, $[2]\leq 2$.
\end{cor}
\begin{proof}
$q^{n+1}+q^{-n-1}=[n+2]-[n]=\sigma_A+\sigma_A^{-1}\in[-2,2]$ only if $q+q^{-1}=[2]\leq 2$.
\end{proof}

\begin{rem}\label{rem:IndexAtMost4}
After suitably labeling the projections at depth $n$, every irreducible subfactor with index at most 4 satisfies the conditions of Theorem \ref{thm:Ocneanu}.
A table of the possible chiralities for these subfactors is given in Example \ref{exs:IndexAtMost4}, as is a proof that $E_7$ and $D_{2k-1}$ ($3\leq k<\I$) are not principal graphs of subfactors. Interestingly, all possible chiralities actually occur.
\end{rem}

We now recover \cite[Theorem 3]{1207.5090} which generalizes \cite[Theorem 5.1.11]{MR2972458}. 
\begin{cor}\label{cor:SinglyValent}
Suppose $[2]>2$,  $(\Gamma_+,\Gamma_-)$ has multiplicity 1, and $P$ is singly valent. Then $n$ is even, $r=[n+2]/[n]$, and
$$
\check{r}+\frac{1}{\check{r}}=2+\frac{\omega_A+\omega_A^{-1}+2}{[n][n+2]}.
$$
Moreover, if $\omega_A$ is a primitive $k$-th root of unity, then $2k\mid n$.
\end{cor}
\begin{proof}
Let $P$ be the singly valent vertex. If $n$ is odd, then the hypotheses of Theorem \ref{thm:Ocneanu} are satisfied, so Corollary \ref{cor:Ocneanu} contradicts $[2]>2$. 
Hence $n$ is even, $\overline{P}=P$ (since $\Tr(P)\neq \Tr(Q)$), and $r=[n+2]/[n]$.
Since $P'=0$, the right hand side in Equation \eqref{eqn:EvenObstructionSelfDual} is zero, and the equation can be rearranged to give
$$
\sigma_A+\sigma_A^{-1}=\sqrt{[n+2][n]}\left( \sqrt{\check{r}} - \frac{1}{\sqrt{\check{r}}}\right).
$$
Squaring both sides and rearranging yields the formula.

For the last statement, consider the 2-click rotation $\rho$ on $\cP_{n,+}$. Since $n=2m$ for some $m$, we have $\rho^m$ is a trace-preserving map on $\spann\{P,Q\}$. Since $\Tr(P)\neq \Tr(Q)$, $\rho^m$ must be the identity, so $\rho^m$ fixes $S$. In particular, $\omega_A^m=1$, so $k\mid m$, and thus $2k\mid n$.
\end{proof}

\begin{rem}
Note that the eigenvalue statement in Corollary \ref{cor:SinglyValent} is essentially in \cite[Theorem 5.1.11]{MR2972458}.
\end{rem}


\begin{thm}\label{thm:MagicNumbers11}
Suppose $(\Gamma_+,\Gamma_-)$ has annular multiplicities $*11$, and $\Gamma_\pm$ does not have a singly valent vertex at depth $n$, i.e., $\Gamma_+,\Gamma_-$ are both translated extensions of 
$$
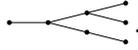
\bigraph{gbg1v1p1v1x0p0x1p0x1}.
$$
\begin{enumerate}[(1)]
\item
If $n$ is even, $(\Gamma_+,\Gamma_-)$ is a translated extension of
$
\left(\bigraph{bwd1v1p1v1x0p0x1p0x1duals1v1x2},\bigraph{bwd1v1p1v1x0p1x0p0x1duals1v1x2}\right)
$,
and
$$
(\check{r}-1)\frac{r}{\check{r}} - \frac{(\sigma_A+\sigma_A^{-1})}{[n]}\frac{\sqrt{r}}{\sqrt{\check{r}}}
=
\frac{r[n]-[n+2]}{[n]}.
$$
\item
If $n$ is odd, $(\Gamma_+,\Gamma_-)$ is a translated extension of
$
\left(\bigraph{bwd1v1v1p1v1x0p0x1p0x1duals1v1v2x1x3},\bigraph{bwd1v1v1p1v1x0p0x1p0x1duals1v1v2x1x3}\right)
$,
and 
$$
(r-1) - \frac{(\sigma_A+\sigma_A^{-1})}{[n]}= \frac{[n+2]-r[n]}{r[n]}.
$$
\end{enumerate}
\end{thm}
\begin{proof}
First, the principal graphs begin as claimed by \cite[Lemmas 6.3 and 6.4]{MR2914056}.
Let $P$ be the bottom vertex at depth $n$ on $\Gamma_+$, and let $P'$ be the bottom vertex at depth $n+1$ on $\Gamma_+$. Let $\check{P}$ be the bottom vertex at depth $n$ on $\Gamma_-$. We record the following traces:
\begin{align*}
\Tr(P) &= \frac{[n+1]}{1+r}
\text{ and }
\Tr(Q) = \frac{r[n+1]}{1+r}
\\
\Tr(\check{P}) &= \frac{[n+1]}{1+\check{r}}
\text{ and }
\Tr(\check{Q}) = \frac{r[n+1]}{1+\check{r}} \text{ provided $n$ is even, and}
\\
\Tr(P')&=[2]\Tr(P)-[n]=[2]\left(\frac{[n+1]}{1+r}\right)-[n]=\frac{[n+2]-r[n]}{1+r}.
\end{align*}
Suppose $n$ is even.
Since $P=\overline{P}$ and $E(\overline{P'})=\check{P}$, the right hand side of Equation \eqref{eqn:EvenObstructionSelfDual} is given by
\begin{align*}
\text{RHS}
&=
-(1+r) \left(\frac{[n+1]}{[n]}\right)\left(\frac{\Tr(P')}{\Tr(\check{P})}\right)\left(\underset{\in \check{P}}{\coeff}\left(\check{A}\right)\right)
\\&=
-(1+r) \left(\frac{[n+1]}{[n]}\right)\left(\frac{[n+2]-r[n]}{1+r}\cdot  \frac{1+\check{r}}{[n+1]}\right)\left(\frac{1}{1+\check{r}} \right)
\\&=
\D\frac{r[n]-[n+2]}{[n]}.
\end{align*}
When $n$ is odd, $E(\overline{P'})=Q$, and the right hand side of Equation \eqref{eqn:OddObstruction} is given by
\begin{align*}
\text{RHS}
&=
-(1+r) \left(\frac{[n+1]}{[n]}\right)\left(\frac{\Tr(P')}{\Tr(Q)}\right)\left(\underset{\in Q}{\coeff}\left(A\right)\right)
\\&=
-(1+r) \left(\frac{[n+1]}{[n]}\right)\left(\frac{[n+2]-r[n]}{1+r}\cdot  \frac{1+r}{r[n+1]}\right)\left(\frac{-1}{1+r} \right)
\\&=
\D\frac{[n+2]-r[n]}{r[n]}.
\end{align*}
\end{proof}

\begin{rem}
Snyder can produce the equations in Theorem \ref{thm:MagicNumbers11} using his technique in \cite{1207.5090}, but his technique does not generalize beyond the $*11$ case. 
\end{rem}

\subsection{Finding chiralities of some examples}

\begin{exs}[Index at most 4]\label{exs:IndexAtMost4}
We now use Theorem \ref{thm:Ocneanu} to determine the chiralities of all non type-$A$ irreducible subfactors with index at most 4 with an initial triple point, and we show $E_7$ and $D_{2k+1}$ are not principal graphs of subfactors. 

Note that the computation for index less than 4 was done previously in \cite[Theorem 4.2.13]{math/9909027} using a different method. Jones' proof of the nonexistence of $D_{2k-1}$ and $E_7$ also shows the chirality is inconsistent with the supertransitivity, but he uses a different formula.

In the table below, $4\leq k$ and $3\leq j\leq \I$.
$$
\begin{array}{|c|c|c|c|c|c|}
\hline
\text{Graph} & n &q & \sigma_A+\sigma_A^{-1} & \text{possible $\sigma_A$} & \text{possible $\omega_A$}
\\\hline
D_{k} & k-2 & \exp\left(\frac{2\pi i}{4k-4}\right)  & 0 & \pm i & -1 \hspace{.2in}(\star\, k\text{ odd})
\\\hline
E_6 & 3 & \exp\left(\frac{2\pi i}{24}\right)  & 1 & \exp\left(\pm \frac{2\pi i}{6}\right) &\exp\left(\pm\frac{2\pi i}{3}\right)
\\\hline
E_7 & 4 & \exp\left(\frac{2\pi i}{36}\right) & 
 2 \sin \left(\frac{2 \pi }{9}\right) & \exp\left(\pm\frac{5 \pi i}{18}\right) & \exp\left(\pm\frac{5 \pi i}{9}\right) \hspace{.2in} \star
\\\hline
E_8 & 5 &\exp\left(\frac{2\pi i}{60}\right)  & \frac{1}{2}(1+\sqrt{5}) &\exp\left(\pm\frac{ 2\pi i}{10}\right) & \exp\left(\pm\frac{2\pi i}{5}\right)
\\\hline
D_{j+2}^{(1)} & 2& 1 &  2 & 1 & 1
\\\hline
E_6^{(1)} & 3 & 1 & 2 & 1 & 1
\\\hline
E_7^{(1)} & 4 & 1 & 2 & 1 & 1
\\\hline
E_8^{(1)} & 6 & 1 & 2 & 1 & 1
\\\hline
\end{array}
$$
The $\star$'s above denote contradictions.
\begin{itemize}
\item
For $D_{k}$ with $k$ odd, $-1$ is not a $(k-2)$-th root of unity.
\item
For $E_7$, $\exp\left(\pm\frac{5 \pi i}{9}\right)$ is not a 4-th root of unity.
\end{itemize}
\end{exs}


\begin{ex}[Fuss-Catalan]
For generic $a,b>2$, the principal graphs of the Bisch-Jones Fuss-Catalan subfactor planar algebras \cite{MR1437496} are extensions of
$$
\left(\bigraph{bwd1v1p1v1x0p0x1p0x1duals1v1x2},\bigraph{bwd1v1p1v1x0p1x0p0x1duals1v1x2}\right).
$$
For these planar algebras, $\omega_A=1$ \cite[Example 3.5.9]{MR2972458}, as can be verified by our formula. 
The traces of the projections labelled as in the proof of Theorem \ref{thm:MagicNumbers11} are given by
\begin{align*}
\Tr(P) &= a^2-1 
&
\Tr(Q) &= a^2(b^2-1)
\\
\Tr(\check{P}) &= b^2(a^2-1) 
&
\Tr(\check{Q}) &= b^2-1
\\
\Tr(P')&=b(a^3-2a).
\end{align*}
Hence the branch factors are given by
$$
r=\frac{a^2(b^2-1)}{a^2-1}
\text{ and }
\check{r} =\frac{b^2-1}{b^2(a^2-1)}
\Longrightarrow 
\frac{r}{\check{r}}=a^2b^2.
$$
Noting that $[n]=[2]=ab$ and $[n+2]=[4]=ab (-2+ab^2)$, we substitute these values into (1) from Theorem \ref{thm:MagicNumbers11} which gives
\begin{align*}
-(\sigma_A+\sigma_A^{-1})
&=
\left(\frac{r[n]-[n+2]}{[n]}-(\check{r}-1)\frac{r}{\check{r}}\right)\frac{\sqrt{\check{r}}}{\sqrt{r}}[n]
\\&=
\left(\frac{\left(\frac{a^2(b^2-1)}{a^2-1}\right)ab-ab (-2+a^2b^2)}{ab}-\left(\frac{b^2-1}{b^2(a^2-1)}-1\right)a^2b^2\right)\left(\frac{1}{ab}\right)ab
\\&=
\left(\frac{a^2(b^2-1)}{a^2-1}\right)+2-a^2b^2-\left( \frac{a^2(b^2-1)}{a^2-1}\right)+a^2b^2
\\&=2.
\end{align*}
Hence $\sigma_A=-1$, and $\omega_A=1$ as claimed.
\end{ex}

\begin{ex}[$U_q({\mathfrak{su}}(3))$]
The principal graphs of the $U_q({\mathfrak{su}}(3))$ subfactors \cite{MR936086} are extensions of
$$
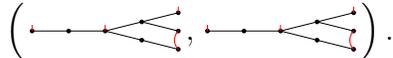
\left(\bigraph{bwd1v1v1p1v1x0p0x1p0x1duals1v1v2x1x3},\bigraph{bwd1v1v1p1v1x0p0x1p0x1duals1v1v2x1x3}\right).
$$
For these planar algebras, $\omega_A=1$, which follows from Kuperberg's $A_2$ spider \cite{MR1403861}, where the generator at depth 3 is a hexagon with legs alternating in and out. Hence it has a $\Z/3$ symmetry. The author would like to thank Scott Morrison for pointing this out.
We verify this with our formula for two $U_q({\mathfrak{su}}(3))$ subfactors:
\begin{align*}
\cQ_1 &=
\left(
\bigraph{bwd1v1v1p1v1x0p1x0p0x1v1x0x0p0x1x1duals1v1v3x2x1},
\bigraph{bwd1v1v1p1v1x0p1x0p0x1v1x0x0p0x1x1duals1v1v3x2x1}
\right)
\\
\cQ_2 &=
\left(
\bigraph{bwd1v1v1p1v1x0p0x1p0x1v0x1x0p1x0x1p0x0x1v0x1x0p1x0x1v1x0duals1v1v2x1x3v2x1},
\bigraph{bwd1v1v1p1v1x0p0x1p0x1v0x1x0p1x0x1p0x0x1v0x1x0p1x0x1v1x0duals1v1v2x1x3v2x1}
\right).
\end{align*}
For $\cQ_1$, the branch factor is the root of $x^3-4 x^2+3 x+1$ which is approximately $1.45$, and $q$ is the root of $x^6-2 x^5+2 x^4-3 x^3+2 x^2-2 x+1$ which is approximately $1.636$. 
For $\cQ_2$, the branch factor is $(2 +\sqrt{2})/2$ and $q=\frac{1}{2} \left(1+\sqrt{2}+\sqrt{2 \sqrt{2}-1}\right)$. 
Using these values in (2) from Theorem \ref{thm:MagicNumbers11}, we get $\sigma_A+\sigma_A^{-1}=-2$.
\end{ex}

\begin{prop}[2D2]\label{prop:2D2}
The 2D2 subfactor planar algebra with principal graphs
$$
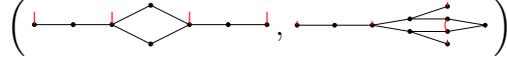
\left(
\bigraph{bwd1v1v1p1v1x1v1v1duals1v1v1v1},
\bigraph{bwd1v1v1p1v1x0p1x0p0x1p0x1v0x1x1x0duals1v1v1x3x2x4}
\right)
$$
has chirality $\omega_A=1$.
\end{prop}
\begin{proof}
We work with the dual subfactor planar algebra with principal graphs 
$$
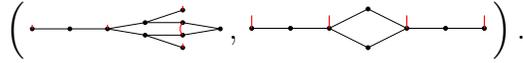
\left(
\bigraph{bwd1v1v1p1v1x0p1x0p0x1p0x1v0x1x1x0duals1v1v1x3x2x4},
\bigraph{bwd1v1v1p1v1x1v1v1duals1v1v1v1}
\right).
$$
Let $P$ be the bottom vertex at depth 3, and let $P_1',P_2'$ be the bottom two vertices at depth 4, where $P_1'$ is below $P_2'$. The traces of $P,P_1',P_2'$ are given by
$$
\Tr(P)=\Tr(Q)=\sqrt{7 + 3 \sqrt{5}},\, \Tr(P_1')=\frac{1}{2} (1 + \sqrt{5}), \text{ and } \Tr(P_2')=\frac{1}{2} (3 + \sqrt{5}).
$$
The right hand side in Equation \eqref{eqn:OddObstruction} is given by
\begin{align*}
\text{RHS}
&=
-(1+r)\left(\frac{[n+1]}{[n]}\right)\left( 
\left(\frac{\Tr(P_1')}{\Tr(P)}\right)\underset{\in P}{\coeff}\left(A\right)
+
\left(\frac{\Tr(P_2')}{\Tr(Q)}\right)\underset{\in Q}{\coeff}\left(A\right)
\right)
\\&=
-2\left(\frac{[n+1]}{[n]}\right)\left( 
\left(\frac{\Tr(P_1')}{\Tr(P)}\right)\left(\frac{1}{2}\right)
-
\left(\frac{\Tr(P_2')}{\Tr(Q)}\right)\left(\frac{1}{2}\right)
\right)
\\&=
\frac{2}{[n]}.
\end{align*}
Since the left hand side is $\D-\frac{\sigma_A+\sigma_A^{-1}}{[n]}$, we have $\sigma_A+\sigma_A^{-1}=-2$, so $\omega_A=1$.
\end{proof}

\begin{prop}
The 2D2 subfactor planar algebra is singly generated at depth $3$. More precisely, any $3$-box not in $\TL_{3,\pm}$ generates all of 2D2.
\end{prop}
\begin{proof}
Since 2D2 has annular multiplicities $*12$, if $A$ generated a proper planar subalgebra, then the planar subalgebra would have either annular multiplicities $*10$ or $*11$. However, $r=\check{r}=1$, which contradicts Corollary \ref{cor:SinglyValent} and Theorem \ref{thm:MagicNumbers11}, since we know $[2]=\sqrt{3+\sqrt{5}}> 2$.
\end{proof}

\subsection{Eliminating a $*11$ weed above index 5}\label{sec:11Weeds}

The following weed with annular multiplicities $*11$ appears when running the principal graph odometer \cite{MR2914056} above index $5$:
$$
\cW=\left(\bigraph{bwd1v1v1p1v0x1p1x0p1x0v1x0x0p1x0x0p0x1x0p0x0x1v1x0x0x0p0x1x0x0p0x0x1x0p0x0x1x0p0x0x0x1v1x0x0x0x0p0x1x0x0x0p0x0x0x1x0p0x0x0x0x1p0x0x0x0x1duals1v1v2x1x3v1x3x2x5x4}, \bigraph{bwd1v1v1p1v1x0p1x0p0x1v1x0x0p0x0x1p0x1x0p0x0x1v1x0x0x0p0x1x0x0p0x0x1x0p0x0x1x0p0x0x0x1v0x0x1x0x0p1x0x0x0x0p0x0x0x0x1p1x0x0x0x0p0x1x0x0x0duals1v1v1x3x2v3x4x1x2x5}\right).
$$
Note that $\cW$ begins as claimed in Theorem \ref{thm:MagicNumbers11}, after applying a suitable graph automorphism.

We now eliminate $\cW$ using Theorem \ref{thm:MagicNumbers11} and the technique of \cite[Section 4.1]{MR2902285}. 
First, we determine the \underline{relative dimensions} of the vertices as functions of $n$ and $q$, where $q>1$ such that $[2]=q+q^{-1}$. 
Second, we calculate the \underline{relative branch factors}, which are the expressions for $r,\check{r}$ as functions of $n,q$.
Finally, we use Theorem \ref{thm:MagicNumbers11} to write $\sigma_A+\sigma_A^{-1}+2\in [0,4]$ as a function of $n,q$, and we show that this value is always negative, a contradiction. 

\begin{thm}\label{thm:EliminateW1}
There is no subfactor with principal graphs a translated extension of $\cW$.
\end{thm}
\begin{proof}
Suppose we had such a subfactor, and note that $n\geq 3$ and $d\geq 2.27453$, so $q\geq 1.6789$.
Let $P$ be the 2-valent vertex on $\Gamma_+$ at depth $n$, and let $\check{P}$ be the bottom vertex on $\Gamma_-$ at depth $n$.
The formula in (2) from Theorem \ref{thm:MagicNumbers11} rearranges to give
\begin{equation}\label{eqn:OddMagic11}
\sigma_A+\sigma_A^{-1}=r[n]-\frac{[n+2]}{r},
\end{equation}
and the relative branch factors are given by
\begin{align*}
r=
\check{r}=
\frac{\left(q^2+1\right) \left(q^{2 n+8}+3 q^{2 n+10}+2 q^{2 n+12}+2 q^{2 n+14}-q^6-3 q^4-2 q^2-2\right)}{2 q^{2 n+8}+4 q^{2 n+10}+4 q^{2 n+12}+4 q^{2 n+14}+q^{2 n+16}-2 q^8-4 q^6-4 q^4-4 q^2-1}.
\end{align*}
We then have 
$$\sigma_A+\sigma_A^{-1}+2=
r[n]-\frac{[n+2]}{r}+2
=\D \frac{g(n,q)h(n,q)}{k(n,q)},$$ 
where $g,h,k$ are given by
\begin{align*}
g(n,q) 
=&
a^3 \left(q^{18}+2 q^{17}+4 q^{16}+4 q^{15}+4 q^{14}+5 q^{13}+4 q^{12}+4 q^{11}+2 q^{10}+q^9\right)
\\&+a^2 \left(-2 q^{17}-q^{16}-4 q^{15}-4 q^{14}-5 q^{13}-4 q^{12}-4 q^{11}-4 q^{10}-q^9-2 q^8\right)
\\&+a \left(-2 q^{10}-q^9-4 q^8-4 q^7-4 q^6-5 q^5-4 q^4-4 q^3-q^2-2 q\right)
\\&+(q^9+2 q^8+4 q^7+4 q^6+5 q^5+4 q^4+4 q^3+4 q^2+2 q+1)
\\
h(n,q)
=&
a^3 \left(q^{18}-2 q^{17}+4 q^{16}-4 q^{15}+4 q^{14}-5 q^{13}+4 q^{12}-4 q^{11}+2 q^{10}-q^9\right)
\\&+a^2 \left(-2 q^{17}+q^{16}-4 q^{15}+4 q^{14}-5 q^{13}+4 q^{12}-4 q^{11}+4 q^{10}-q^9+2 q^8\right)
\\&+a \left(-2 q^{10}+q^9-4 q^8+4 q^7-4 q^6+5 q^5-4 q^4+4 q^3-q^2+2 q\right)
\\&+(q^9-2 q^8+4 q^7-4 q^6+5 q^5-4 q^4+4 q^3-4 q^2+2 q-1)
\\
k(n,q)=&
-q^{n+1}(q-1) (q+1) \left(q^2+1\right) 
\\&\times\left(a^{2}(2 q^{14}+2 q^{12}+3 q^{10}+q^{8})-q^6-3 q^4-2 q^2-2\right)
\\&\times \left(a^2(q^{16}+4 q^{14}+4 q^{12}+4 q^{10}+2 q^{8})-2 q^8-4 q^6-4 q^4-4 q^2-1\right)
\end{align*}
using the shorthand $a=q^n$.
Note that $k(n,q)$ is always negative.
\begin{claim*}
For all $n\geq 0$ and $q\geq 1.6789$, $g(n,q)\geq 0$.
\end{claim*}
\begin{proof}[Proof of Claim]
We break up $g(n,q)$ into four parts: $g(n,q)=\sum_{j=0}^3 a^j g_j(q)$ where
\begin{align*}
g_3(q) &=
q^{18}+2 q^{17}+4 q^{16}+4 q^{15}+4 q^{14}+5 q^{13}+4 q^{12}+4 q^{11}+2 q^{10}+q^9\\
g_2(q) &=
-2 q^{17}-q^{16}-4 q^{15}-4 q^{14}-5 q^{13}-4 q^{12}-4 q^{11}-4 q^{10}-q^9-2 q^8\\
g_1(q) &=
-2 q^{10}-q^9-4 q^8-4 q^7-4 q^6-5 q^5-4 q^4-4 q^3-q^2-2 q\\
g_0(q) &=
q^9+2 q^8+4 q^7+4 q^6+5 q^5+4 q^4+4 q^3+4 q^2+2 q+1
\end{align*}
One checks that for $q\geq 1.6789$, 
\begin{align*}
0&< g_0(q)\\
0&< g_3(q)+g_2(q)+g_1(q)\\
0&< g_3(q)+g_2(q)\text{ and }\\
0&< g_3(q).
\end{align*}
Hence for $a=q^n$ with $q\geq 1.6789$ and $n\geq 0$, we have 
\begin{align*}
0 &< 
g_3(q)+g_2(q)+g_1(q)+g_0(q)
\\&\leq 
a (g_3(q)+ g_2(q)+g_1(q))+g_0(q)
\\&\leq 
a (a(g_3(q)+ g_2(q))+g_1(q))+g_0(q)
\\&\leq 
a (a(ag_3(q)+ g_2(q))+g_1(q))+g_0(q)
\\&=g(n,q).
\end{align*}
\end{proof}

\begin{claim*}
For $n\geq 2$ and  $q\geq 1.6789$, $h(n,q)\geq 0$.
\end{claim*}
\begin{proof}[Proof of Claim]
Since $n\geq 2$, we replace $n$ with $m=n+2$. We break up $h(m+2,q)$ into four parts: $h(m+2,q)=\sum_{j=0}^3 b^j h_j(q)$ where
\begin{align*}
h_3(q) &=
q^{24}-2 q^{23}+4 q^{22}-4 q^{21}+4 q^{20}-5 q^{19}+4 q^{18}-4 q^{17}+2 q^{16}-q^{15}\\
h_2(q) &=
-2 q^{21}+q^{20}-4 q^{19}+4 q^{18}-5 q^{17}+4 q^{16}-4 q^{15}+4 q^{14}-q^{13}+2 q^{12}\\
h_1(q) &=
-2 q^{12}+q^{11}-4 q^{10}+4 q^9-4 q^8+5 q^7-4 q^6+4 q^5-q^4+2 q^3\\
h_0(q) &=
q^9-2 q^8+4 q^7-4 q^6+5 q^5-4 q^4+4 q^3-4 q^2+2 q-1
\end{align*}
and $b=q^m$. The rest is identical to the previous claim replacing $g$'s with $h$'s, $a$ with $b$, and $n$ with $m$ (except that the final conclusion is that $h(m+2,q)>0$).
\end{proof}
Thus $\D 0\leq \sigma_A+\sigma_A^{-1}+2= \frac{g(n,q)h(n,q)}{k(n,q)}<0$ for $n\geq 2$, a contradiction.
\end{proof}

\section{Multiplicity 2 - some quadruple points}\label{sec:Multiplicity2}
Suppose $(\Gamma_+,\Gamma_-)$ has multiplicity $2$, $n\geq 2$ is even, and two of the vertices at depth $n$ of $\Gamma_\pm$ are dual to each other, i.e., $(\Gamma_+,\Gamma_-)$ is an even translated extension of
$$
\cQ=\left(\bigraph{bwd1v1p1p1duals1v1x3x2},\bigraph{bwd1v1p1p1duals1v1x3x2}\right).
$$
For examples of subfactors with such principal graphs, see Examples \ref{exs:Index6}.
We use the following conventions:
\begin{itemize}
\item $P,Q,R$ and $\check{P},\check{Q},\check{R}$ are the projections at depth $n$ of $\Gamma_\pm$ from bottom to top,
\item $\D r= \frac{2\Tr(Q)}{\Tr(P)}$,
\item  $A=rP-(Q+R)$ and $B=Q-R\in \cP_{n,+}$ are low-weight rotational eigenvectors with distinct eigenvalues $\omega_A,\omega_B$ (by Proposition \ref{prop:LowWeight} below),
\item row reducing 
$
\left(
\begin{array}{ccc|c}
r & -1 & -1 & A
\\
0 & -1 & 1 & B
\\
1 & 1 & 1 & \jw{n}
\end{array}
\right)
$
yields
$$
\begin{pmatrix}
P
\\
Q
\\
R
\end{pmatrix}
=
\D
\frac{1}{2(1+r)}
\begin{pmatrix}
2A+2\jw{n}
\\
-A-(1+r)B+r\jw{n}
\\
-A+(1+r)B+r\jw{n}
\end{pmatrix},
$$
\item
Substituting the formulas for $P,Q,R$ yields
$$
\begin{pmatrix}
A^2
\\
AB
\\
B^2
\end{pmatrix}
=
\begin{pmatrix}
r^2P+Q+R
\\
R-Q
\\
Q+R
\end{pmatrix}
=
\D
\begin{pmatrix}
(r-1)A+r\jw{n}
\\
-B
\\
\frac{-1}{1+r}A+\frac{r}{1+r}\jw{n}
\end{pmatrix},
$$

\item the previous 4 lines have analogous formulas with checks, and
\item $\D\cF(A)=\frac{\sqrt{r}}{\sqrt{\check{r}}}\sigma_A  \check{A}$ and $\D\cF(B)=\frac{\sqrt{r(1+\check{r})}}{\sqrt{\check{r}(1+r)}}\sigma_B  \check{B}$ (also by Proposition \ref{prop:LowWeight}).
\end{itemize}
The next result is a modified version of \cite[Lemma 2.1]{MR2993924}. We provide a short proof for convenience and completeness.

\begin{prop}\label{prop:LowWeight}
The elements $A=rP-(Q+R)$ and $B=Q-R$ are rotational low weight eigenvectors such that $\omega_A^{n/2}=1$ and $\omega_B^{n/2}=-1$.
The same holds with checks, and thus there are completely determined $\sigma_A,\sigma_B$ such that $\sigma_A^2=\omega_A, \sigma_B^2=\omega_B$, and the above formulas hold for $\cF(A),\cF(B)$.
\end{prop}
\begin{proof}
First, note that $\spann\{A,B\}=\TL_{n,+}^\perp$, since $A,B$ are nonzero, orthogonal, and have trace zero.
Since $P$ is self-dual and $Q$ is dual to $R$,
\begin{align*}
\rho^{n/2}(A) &=\rho^{n/2}(rP-(Q+R))=rP-(Q+R)=A\text{ and}\\
\rho^{n/2}(B) &= \rho^{n/2}(Q-R)=R-Q=-B.
\end{align*}
Thus $\{A,B\}$ is a basis of eigenvectors for $\rho^{n/2}$ on $\TL_{n,+}^\perp$.
Let $\{v_1,v_2\}$ be a basis of eigenvectors for $\rho$ on $\TL_{n,+}^\perp$. 
Then $\{v_1,v_2\}$ is a basis of eigenvectors for $\rho^{n/2}$, so up to scaling, we must have $v_1=A$ and $v_2=B$.
The same holds with checks. 

Since we know $\omega_A\neq \omega_B$, we know $\cF(A)=\sigma_A\lambda_A \check{A}$ for some real scalar $\lambda_A$ and a completely determined $\sigma_A$, and similar for $B$. Taking the norm squared of each side gives the value of $\lambda_A,\lambda_B$.
\end{proof}

Since $P,Q$ only connect to $\jw{n-1}$ at depth $n-1$, we may define $P',Q'$ and $\cap_{n+1}(\overline{P'}),\cap_{n+1}(\overline{Q'})$ as in Definition \ref{defn:PPrime} applied to $P,Q$ respectively.

\begin{thm}\label{thm:QuadruplePoint}
We have the following equations:
\begin{align}
(\check{r}-1)\frac{r}{\check{r}} 
- 
\frac{(\sigma_A+\sigma_A^{-1})}{[n]}\frac{\sqrt{r}}{\sqrt{\check{r}}}
=&
-(1+r) \frac{[n+1]}{[n]} \left(\underset{\in \cap_{n+1}(\overline{P'})}{\coeff}\left(\check{A}\right)\right)
\tag{$\cQ$A1}\label{eqn:QA1}
\\
\left(
\left(\check{r}-r-2\right)
+
\frac{(\sigma_A+\sigma_A^{-1})}{[n]}\sqrt{r\check{r}}
\right)
\frac{r}{\check{r}}
=&
-2r(1+r) \frac{[n+1]}{[n]} \left(\underset{\in \cap_{n+1}(\overline{Q'})}{\coeff}\left(\check{A}\right)\right)
\tag{$\cQ$A2}\label{eqn:QA2}
\\
\left(
(\sigma_A\sigma_B^{-1}+\sigma_A^{-1}\sigma_B)
-
\frac{(\sigma_B+\sigma_B^{-1})}{[n]}
\right.&\left.
\sqrt{r\check{r}}
\right)
\left(\frac{r}{\check{r}}\right)\sqrt{1+\check{r}}\sqrt{1+r}
\notag\\
=&
-2r(1+r) \frac{[n+1]}{[n]} \left(\underset{\in \cap_{n+1}(\overline{Q'})}{\coeff}\left(\check{B}\right)\right)
\tag{$\cQ$B}\label{eqn:QB}
\end{align}
\end{thm}
\begin{proof}
We find the coefficient of $\check{A},\check{B}$ in both sides of Liu's relation, noting that $\check{A},\check{B}$ are orthogonal to $\TL_{n,-}$ and are orthogonal to each other. 

Since $\D P=\frac{\jw{n}+ A}{1+r}$, the proof of Equation \eqref{eqn:QA1} is identical to the proof of Equation \eqref{eqn:EvenObstructionSelfDual} in Theorem \ref{thm:TriplePoint}.

For Equations \eqref{eqn:QA2} and \eqref{eqn:QB}, we note that $\D \overline{Q}=R=\frac{r\jw{n}-A+(1+r)B}{2(1+r)}$, so expanding the left hand side of Liu's relation applied to $Q$ gives 
$$
\cF(R)^*\cF(R)
=
\frac{1}{4(1+r)^2}\cF(r\jw{n}-A+(1+r)B)^*\cF(r\jw{n}-A+(1+r)B).
$$
This time there are 9 terms, 8 of which contribute some multiple of $\check{A}$ or $\check{B}$.
These terms are as follows, where $T_A,T_B\in \TL_{n,-}$, and we omit the leading coefficient of $\D\frac{1}{4(1+r)^2}$, which we will move to the other side of the equation:
\begin{align*}
\begin{tikzpicture}[baseline=-.1cm]
	\draw (0,-1.4)--(0,1.4);
	\node at (0,1.6) {\scriptsize{$n-1$}};
	\node at (0,-1.6) {\scriptsize{$n-1$}};
	\draw (.2,1.1) arc (180:0:.2cm) -- (.6,-1.1) arc (0:-180:.2cm);
	\draw (-.2,.3) arc (0:-180:.2cm) -- (-.6,1.6);
	\draw (-.2,-.3) arc (0:180:.2cm) -- (-.6,-1.6);
	\nbox{unshaded}{(0,-.7)}{0}{0}{A}
	\nbox{unshaded}{(0,.7)}{0}{0}{A}
\end{tikzpicture}
&=
\frac{r}{\check{r}}\check{A}^2 
= 
\left(\frac{r(\check{r}-1)}{\check{r}}\right)\check{A}+ T_A
\displaybreak[1]\\
(1+r)^2\,
\begin{tikzpicture}[baseline=-.1cm]
	\draw (0,-1.4)--(0,1.4);
	\node at (0,1.6) {\scriptsize{$n-1$}};
	\node at (0,-1.6) {\scriptsize{$n-1$}};
	\draw (.2,1.1) arc (180:0:.2cm) -- (.6,-1.1) arc (0:-180:.2cm);
	\draw (-.2,.3) arc (0:-180:.2cm) -- (-.6,1.6);
	\draw (-.2,-.3) arc (0:180:.2cm) -- (-.6,-1.6);
	\nbox{unshaded}{(0,-.7)}{0}{0}{B}
	\nbox{unshaded}{(0,.7)}{0}{0}{B}
\end{tikzpicture}
&=
(1+r)^2\frac{r(1+\check{r})}{\check{r}(1+r)}\check{B}^2 
= 
\left(-\frac{r(1+r)}{\check{r}}\right)\check{A} + T_B
\displaybreak[1]\\
-r\left(
\begin{tikzpicture}[baseline=-.1cm]
	\draw (0,-1.4)--(0,1.4);
	\node at (0,1.6) {\scriptsize{$n-1$}};
	\node at (0,-1.6) {\scriptsize{$n-1$}};
	\draw (.2,1.1) arc (180:0:.2cm) -- (.6,-1.1) arc (0:-180:.2cm);
	\draw (-.2,.3) arc (0:-180:.2cm) -- (-.6,1.6);
	\draw (-.2,-.3) arc (0:180:.2cm) -- (-.6,-1.6);
	\nbox{unshaded}{(0,-.7)}{0}{0}{\jw{n}}
	\nbox{unshaded}{(0,.7)}{0}{0}{A}
\end{tikzpicture}
+
\begin{tikzpicture}[baseline=-.1cm]
	\draw (0,-1.4)--(0,1.4);
	\node at (0,1.6) {\scriptsize{$n-1$}};
	\node at (0,-1.6) {\scriptsize{$n-1$}};
	\draw (.2,1.1) arc (180:0:.2cm) -- (.6,-1.1) arc (0:-180:.2cm);
	\draw (-.2,.3) arc (0:-180:.2cm) -- (-.6,1.6);
	\draw (-.2,-.3) arc (0:180:.2cm) -- (-.6,-1.6);
	\nbox{unshaded}{(0,-.7)}{0}{0}{A}
	\nbox{unshaded}{(0,.7)}{0}{0}{\jw{n}}
\end{tikzpicture}
\right)
&= 
\left(-\frac{r(-1)^{n-1}}{[n]}\frac{\sqrt{r}}{\sqrt{\check{r}}}(\sigma_A+\sigma_A^{-1})\right)\check{A} 
\displaybreak[1]\\
r(1{+}r)\left(
\begin{tikzpicture}[baseline=-.1cm]
	\draw (0,-1.4)--(0,1.4);
	\node at (0,1.6) {\scriptsize{$n-1$}};
	\node at (0,-1.6) {\scriptsize{$n-1$}};
	\draw (.2,1.1) arc (180:0:.2cm) -- (.6,-1.1) arc (0:-180:.2cm);
	\draw (-.2,.3) arc (0:-180:.2cm) -- (-.6,1.6);
	\draw (-.2,-.3) arc (0:180:.2cm) -- (-.6,-1.6);
	\nbox{unshaded}{(0,-.7)}{0}{0}{\jw{n}}
	\nbox{unshaded}{(0,.7)}{0}{0}{B}
\end{tikzpicture}
+
\begin{tikzpicture}[baseline=-.1cm]
	\draw (0,-1.4)--(0,1.4);
	\node at (0,1.6) {\scriptsize{$n-1$}};
	\node at (0,-1.6) {\scriptsize{$n-1$}};
	\draw (.2,1.1) arc (180:0:.2cm) -- (.6,-1.1) arc (0:-180:.2cm);
	\draw (-.2,.3) arc (0:-180:.2cm) -- (-.6,1.6);
	\draw (-.2,-.3) arc (0:180:.2cm) -- (-.6,-1.6);
	\nbox{unshaded}{(0,-.7)}{0}{0}{B}
	\nbox{unshaded}{(0,.7)}{0}{0}{\jw{n}}
\end{tikzpicture}
\right)
&= 
\left(\frac{r(1+r)(-1)^{n-1}}{[n]}\frac{\sqrt{r(1+\check{r})}}{\sqrt{\check{r}(1+r)}}(\sigma_B+\sigma_B^{-1})\right)\check{B}
\displaybreak[1]\\
-(1+r)\left(
\begin{tikzpicture}[baseline=-.1cm]
	\draw (0,-1.4)--(0,1.4);
	\node at (0,1.6) {\scriptsize{$n-1$}};
	\node at (0,-1.6) {\scriptsize{$n-1$}};
	\draw (.2,1.1) arc (180:0:.2cm) -- (.6,-1.1) arc (0:-180:.2cm);
	\draw (-.2,.3) arc (0:-180:.2cm) -- (-.6,1.6);
	\draw (-.2,-.3) arc (0:180:.2cm) -- (-.6,-1.6);
	\nbox{unshaded}{(0,-.7)}{0}{0}{B}
	\nbox{unshaded}{(0,.7)}{0}{0}{A}
\end{tikzpicture}
+
\begin{tikzpicture}[baseline=-.1cm]
	\draw (0,-1.4)--(0,1.4);
	\node at (0,1.6) {\scriptsize{$n-1$}};
	\node at (0,-1.6) {\scriptsize{$n-1$}};
	\draw (.2,1.1) arc (180:0:.2cm) -- (.6,-1.1) arc (0:-180:.2cm);
	\draw (-.2,.3) arc (0:-180:.2cm) -- (-.6,1.6);
	\draw (-.2,-.3) arc (0:180:.2cm) -- (-.6,-1.6);
	\nbox{unshaded}{(0,-.7)}{0}{0}{A}
	\nbox{unshaded}{(0,.7)}{0}{0}{B}
\end{tikzpicture}
\right)
&=
\left(
(1+r)
\frac{\sqrt{r(1+\check{r})}}{\sqrt{\check{r}(1+r)}}
\frac{\sqrt{r}}{\sqrt{\check{r}}}
(\sigma_A\sigma_B^{-1}+\sigma_A^{-1}\sigma_B)
\right)
\check{B}.
\end{align*}
Now collect the coefficients of $\check{A}$ and $\check{B}$ and substitute for $\Tr(Q)$.
\end{proof}

\begin{rems}
\mbox{}
\begin{enumerate}[(1)]
\item
Once again, we need criteria similar to the criterion given in Remark \ref{rem:Calculate} to determine $\cap_{n+1}(\overline{P'})$ and $\cap_{n+1}(\overline{Q'})$.
\item
In specific examples, one can first use Equations \eqref{eqn:QA1} and \eqref{eqn:QA2} to compute $\sigma_A$, and then use Equation \eqref{eqn:QB} to compute $\sigma_B$. We will compute some examples in the next subsection.
\item
Having two equations for $\sigma_A+\sigma_A^{-1}$ should give a strong constraint for translated extensions of $\cQ$.
We will see in Subsection \ref{sec:EliminateQuadruple} that we only need Equation \eqref{eqn:QA1} to eliminate two such weeds.
\end{enumerate}
\end{rems}

\subsection{Checking formulas on group-like subfactors at index 6}

We now check Equations \eqref{eqn:QA1}, \eqref{eqn:QA2}, and \eqref{eqn:QB} on some group-like subfactors at index 6. For each of these examples $n=2$, so we know the chiralities are $\omega_A=1$ and $\omega_B=-1$ by Proposition \ref{prop:LowWeight}.

\begin{exs}\label{exs:Index6}
There are several group-like subfactors at index 6 which are translated extensions of $\cQ$.
Details on computing these principal graphs can be found in \cite{MR1386923,MR1738515,Index6}.

First, we consider group-subgroup subfactors at index 6, which can be classified by subgroups $G$ of $S_6$ which act transitively on $\{1,2,3,4,5,6\}$.
The subgroup $H$ of $G$ is the point stabilizer $\Stab(1)$ of the action.
We consider the following subgroups of $S_6$ which act transitively, where the notation for the groups comes from \cite{SubgroupsOfS6}. 
\begin{itemize}
\item 
$A_4a=\langle (123)(456), (135)(246) \rangle \subset S_6$
yields the principal graphs
$$ 
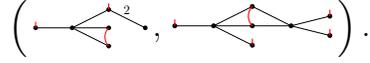
\left(
\bigraph{bwd1v1p1p1v0x0x2duals1v2x1x3}, 
\bigraph{bwd1v1p1p1v0x1x1v1p1duals1v1x3x2v1x2}
\right).
$$
\item $S_3\times 3a=\langle (123)(456), (14)(26)(35), (142635) \rangle \subset S_6$
yields the principal graphs
$$
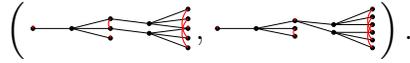
\left(
\bigraph{bwd1v1p1p1v0x1x0p0x0x1v1x0p1x0p1x1p0x1p0x1duals1v1x3x2v4x5x3x1x2}, 
\bigraph{bwd1v1p1p1v0x0x1p0x0x1v1x0p1x0p1x0p0x1p0x1p0x1duals1v2x1x3v4x5x6x1x2x3}
\right).
$$
\item $A_4\times 2a=\langle (123)(456), (135)(246), (123456) \rangle \subset S_6$
yields the principal graphs
$$
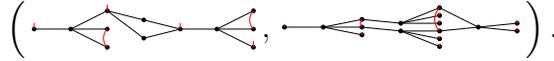
\left(
\bigraph{bwd1v1p1p1v0x0x1p0x0x1v1x1v1v1p1p1duals1v2x1x3v1v1x3x2}, 
\bigraph{bwd1v1p1p1v0x1x0p0x0x1v1x0p1x0p1x0p0x1p0x1p0x1v0x0x1x0x0x1v1p1duals1v1x3x2v1x2x6x4x5x3v1x2}
\right).
$$
\end{itemize}
We also consider the Bisch-Haagerup subfactors $R^H\subset R\rtimes K$ for $H=\Z/2$ and $K=\Z/3$, where as usual, we let $G=\langle H,K\rangle\subset \Out(R)$ \cite{MR1386923}. We consider the following triples $(G,H,K)$:
\begin{itemize}
\item $H=\langle (12)(34)\rangle$, $K=\langle (123)\rangle$, and $G=A_4$
yields the principal graphs
$$ 
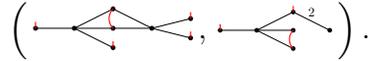
\left(
\bigraph{bwd1v1p1p1v0x1x1v1p1duals1v1x3x2v1x2},
\bigraph{bwd1v1p1p1v0x0x2duals1v2x1x3}
\right).
$$
\item $H=\langle (34)\rangle$, $K=\langle (123)\rangle$, and $G=S_4$
yields the principal graphs
$$
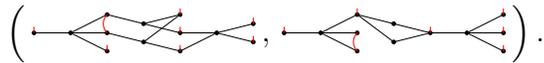
\left(
\bigraph{bwd1v1p1p1v0x1x0p0x0x1v1x0p0x1p1x1v1x1x0v1p1duals1v1x3x2v1x2x3v1x2},
\bigraph{bwd1v1p1p1v0x0x1p0x0x1v1x1v1v1p1p1duals1v2x1x3v1v1x2x3}
\right).
$$
\item $H=\langle (15)(24)\rangle$, $K=\langle (123)\rangle$, and $G=A_5$
yields the principal graphs
$$
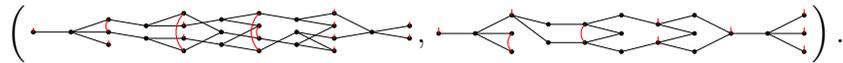
\left(
\bigraph{bwd1v1p1p1v0x1x0p0x0x1v1x0p1x0p0x1p0x1v0x1x0x0p1x0x0x1p0x0x1x0v1x0x0p0x1x0p1x0x1p0x0x1v0x1x0x0p1x0x0x0p0x0x0x1v1x1x0p0x1x0p1x0x1p0x0x1v0x1x0x1v1p1duals1v1x3x2v4x2x3x1v4x3x2x1v1x2x3x4v1x2}
,
\bigraph{bwd1v1p1p1v0x0x1p0x0x1v1x0p0x1v1x0p1x1p0x1v1x0x0p0x0x1v1x0p1x1p0x1v1x0x1v1v1p1p1duals1v2x1x3v2x1v1x2v1v1x2x3}
\right).
$$
\end{itemize}
\end{exs}

In fact, just one calculation suffices to determine the chiralities for Examples \ref{exs:Index6} via our equations. If necessary, we pass to the dual subfactor so we may always assume $P$ is the univalent self-dual vertex at depth 2, i.e., the principal graphs are
$$
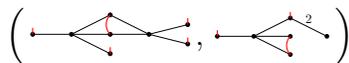
\left(
\bigraph{bwd1v1p1p1v0x1x1v1p1duals1v1x3x2v1x2},
\bigraph{bwd1v1p1p1v0x0x2duals1v2x1x3}
\right)
$$
or they begin like
$$
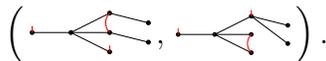
\left(
\bigraph{bwd1v1p1p1v0x1x0p0x0x1duals1v1x3x2}, 
\bigraph{bwd1v1p1p1v0x0x1p0x0x1duals1v2x1x3}
\right).
$$
For these cases, $P'=0$ and $\D\cap_{n+1}(\overline{Q'})=\frac{\Tr(Q')}{\Tr(\check{P})}\check{P}$.
We have the following traces:
\begin{align*}
\Tr(P) &= 1
&
\Tr(Q) &= \Tr(R) = 2
\\
\Tr(\check{P}) &= 3
&
\Tr(\check{Q}) &=\Tr(\check{R}) = 1
\\
&
&
\Tr(Q') &= \Tr(R') = \sqrt{6}.
\end{align*}  
Thus $r=4$ and $\check{r}=\frac{2}{3}$, and Equations \eqref{eqn:QA1}, \eqref{eqn:QA2}, and \eqref{eqn:QB} are given respectively by
\begin{align*}
&\left(\frac{2}{3}-1\right)\frac{12}{2} 
- 
\frac{(\sigma_A+\sigma_A^{-1})}{\sqrt{6}}\frac{2\sqrt{3}}{\sqrt{2}}
=
0,
\\
&\left(
\left(\frac{2}{3}-4-2\right)
+
\frac{(\sigma_A+\sigma_A^{-1})}{\sqrt{6}}2\frac{\sqrt{2}}{\sqrt{3}}
\right)
\frac{12}{2}
=
-2(4)(5) \frac{5}{\sqrt{6}} \left(\frac{\sqrt{6}}{3} \right)\left( \frac{1}{1+\frac{2}{3}}
\right), \text{ and}
\\
&\left(
(\sigma_A\sigma_B^{-1}+\sigma_A^{-1}\sigma_B)
-
\frac{(\sigma_B+\sigma_B^{-1})}{\sqrt{6}}2\frac{\sqrt{2}}{\sqrt{3}}
\right)
\left(\frac{12}{2}\right)\sqrt{1+\frac{2}{3}}\sqrt{5}
=
0.
\end{align*}
Hence $\sigma_A+\sigma_A^{-1}=-2$, so $\sigma_A=-1$ and $\omega_A=1$, and $\sigma_B+\sigma_B^{-1}=0$, so $\sigma_B=\pm i$ and $\omega_B=-1$.

\subsection{Quadruple point obstructions}\label{sec:EliminateQuadruple}

We now prove a result similar to Theorem \ref{thm:Ocneanu} for even translated extensions of $\cQ$.

\begin{thm}\label{thm:OcneanuQuadruple}
Suppose that $(\Gamma_+,\Gamma_-)$ is a translated extension of $\cQ$ such that
\begin{itemize}
\item 
$r=\check{r}$, and for every vertex $R$ connected to  $P$, $\overline{R}$ only connects to $\check{P}$.
\end{itemize}
Then $\sigma_A+\sigma_A^{-1}=[n+2]-[n]=q^{n+1}+q^{-n-1}$, and thus $[2]\leq 2$.
\end{thm}
\begin{proof}
Just as for Theorem \ref{thm:Ocneanu}, we have
$$
\Tr(P')=[2]\Tr(P)-[n]=[2]\left(\frac{[n+1]}{1+r}\right)-[n]=\frac{[n+2]-r[n]}{1+r}.
$$
Since Equation \eqref{eqn:QA1} is identical to Equation \eqref{eqn:EvenObstructionSelfDual}, the same proof for the case $n$ even with $\overline{P}=P$ for Theorem \ref{thm:Ocneanu} shows $\sigma_A+\sigma_A^{-1}=[n+2]-[n]$, and thus $[2]\leq 2$ as in Corollary \ref{cor:Ocneanu}.
\end{proof}

\begin{defn}
We define the following weeds with initial quadruple points which occur when running the principal graph odometer \cite{MR2914056} above index $5$:
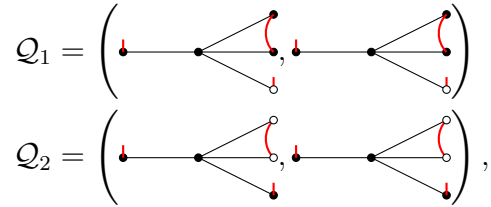
\begin{align*}
\cQ_1&=
\left(
\begin{tikzpicture}[baseline=-.1cm]
	\draw[fill] (0,0) circle (0.05);
	\draw (0.,0.) -- (1.,0.);
	\draw[fill] (1.,0.) circle (0.05);
	\draw (1.,0.) -- (2.,-0.5);
	\draw (1.,0.) -- (2.,0.);
	\draw (1.,0.) -- (2.,0.5);
	\draw[fill] (2.,0.) circle (0.05);
	\draw[fill] (2.,0.5) circle (0.05);
	\draw[red, thick] (0.,0.) -- +(0,0.166667) ;
	\draw[red, thick] (2.,-0.5) -- +(0,0.166667) ;
	\draw[red, thick] (2.,0.) to[out=135,in=-135] (2.,0.5);
	\filldraw[unshaded] (2.,-0.5) circle (0.05);
\end{tikzpicture}
,
\begin{tikzpicture}[baseline=-.1cm]
	\draw[fill] (0,0) circle (0.05);
	\draw (0.,0.) -- (1.,0.);
	\draw[fill] (1.,0.) circle (0.05);
	\draw (1.,0.) -- (2.,-0.5);
	\draw (1.,0.) -- (2.,0.);
	\draw (1.,0.) -- (2.,0.5);
	\draw[fill] (2.,0.) circle (0.05);
	\draw[fill] (2.,0.5) circle (0.05);
	\draw[red, thick] (0.,0.) -- +(0,0.166667) ;
	\draw[red, thick] (2.,-0.5) -- +(0,0.166667) ;
	\draw[red, thick] (2.,0.) to[out=135,in=-135] (2.,0.5);
	\filldraw[unshaded] (2.,-0.5) circle (0.05);
\end{tikzpicture}
\right)
\\
\cQ_2 &=
\left(
\begin{tikzpicture}[baseline=-.1cm]
	\draw[fill] (0,0) circle (0.05);
	\draw (0.,0.) -- (1.,0.);
	\draw[fill] (1.,0.) circle (0.05);
	\draw (1.,0.) -- (2.,-0.5);
	\draw (1.,0.) -- (2.,0.);
	\draw (1.,0.) -- (2.,0.5);
	\draw[fill] (2.,-0.5) circle (0.05);
	\draw[red, thick] (0.,0.) -- +(0,0.166667) ;
	\draw[red, thick] (2.,-0.5) -- +(0,0.166667) ;
	\draw[red, thick] (2.,0.) to[out=135,in=-135] (2.,0.5);
	\filldraw[unshaded] (2.,0.5) circle (0.05);
	\filldraw[unshaded] (2.,0) circle (0.05);
\end{tikzpicture}
,
\begin{tikzpicture}[baseline=-.1cm]
	\draw[fill] (0,0) circle (0.05);
	\draw (0.,0.) -- (1.,0.);
	\draw[fill] (1.,0.) circle (0.05);
	\draw (1.,0.) -- (2.,-0.5);
	\draw (1.,0.) -- (2.,0.);
	\draw (1.,0.) -- (2.,0.5);
	\draw[fill] (2.,-0.5) circle (0.05);
	\draw[red, thick] (0.,0.) -- +(0,0.166667) ;
	\draw[red, thick] (2.,-0.5) -- +(0,0.166667) ;
	\draw[red, thick] (2.,0.) to[out=135,in=-135] (2.,0.5);
	\filldraw[unshaded] (2.,0.5) circle (0.05);
	\filldraw[unshaded] (2.,0) circle (0.05);
\end{tikzpicture}
\right),
\end{align*}
where we use the convention that if a weed has shaded and unshaded vertices, then only the unshaded vertices may connect to a vertex at depth $n+1$.
\end{defn}

\begin{cor}\label{cor:EliminateQ}
Any subfactor whose principal graphs are an even translated extension of $\cQ_1$ or $\cQ_2$ is the $\Z/4$ group subfactor with principal graphs
$$
\left(
\bigraph{bwd1v1p1p1duals1v1x3x2},
\bigraph{bwd1v1p1p1duals1v1x3x2}
\right).
$$
\end{cor}
\begin{proof}
Note that $\cQ_1,\cQ_2$ both satisfy the conditions of Theorem \ref{thm:OcneanuQuadruple}.
By Theorem \ref{thm:OcneanuQuadruple}, for both $\cQ_1,\cQ_2$, $[2]\leq 2$, and any translated extension of $\cQ_1,\cQ_2$ must be the trivial translated extension.
\end{proof}

\begin{rem}
In \cite{MR2914056} Morrison-Snyder showed that the principal graph pair of any non Temperley-Lieb subfactor planar algebra in the index range $(4,5)$ was either a translated extension of one of 5 weeds or a translation of one of 39 vines, up to taking duals. One of their weeds was given by
$$
\cQ_3=\left(\bigraph{bwd1v1v1v1p1p1v1x0x0duals1v1v1x3x2},\bigraph{bwd1v1v1v1p1p1v1x0x0duals1v1v1x3x2}\right),
$$
which is a translated extension of the weed $\cQ_1$. 
(In \cite{MR2914056}, $\cQ_3$ is denoted by $\cQ$, but we have already used this symbol.)
Hence Corollary \ref{cor:EliminateQ} gives another proof that there is no subfactor whose principal graphs are a translated extension of $\cQ_3$.
\end{rem}

\subsection{The other quadruple point when $n$ is even}

For formulas such as those in Theorem \ref{thm:QuadruplePoint} to hold, we did not need to assume that $(\Gamma_+,\Gamma_-)$ was a translated extension of $\cQ$; rather, we needed to know the formulas of the low weight rotational eigenvectors at depth $n$. 
Knowing the dual data ensured that we could determine the new low weight rotational eigenvectors in Proposition \ref{prop:LowWeight}.
If we assumed the same formulas for the low weight rotational eigenvectors for translated extensions of
$$
\left(\bigraph{bwd1v1p1p1duals1v1x2x3},\bigraph{bwd1v1p1p1duals1v1x2x3}\right),
$$
along with $\cF(A)=\sigma_A\lambda_A \check{A}$ and $\cF(B)=\sigma_B\lambda_B \check{B}$ for some real scalars $\lambda_A,\lambda_B$, we would have obtained the same formulas.

Since we cannot verify these assumptions in this case, we cannot use these formulas to prove obstructions to principal graphs. However, we can prove that the low weight rotational eigenvectors are not given by these formulas, as in the following proposition.

\begin{prop}\label{prop:NotBothEigenvectors}
Suppose $(\Gamma_+,\Gamma_-)$ is a translated extension of
$$
\cQ_4=
\left(
\bigraph{bwd1v1p1p1v1x0x0duals1v1x2x3},
\bigraph{bwd1v1p1p1v1x0x0duals1v1x2x3}
\right).
$$
Let $P,Q,R$ and $\check{P},\check{Q},\check{R}$ be the minimal projections at depth 4 from bottom to top.
At least one of $Q-R$ or $\check{Q}-\check{R}$ is \underline{not} a low weight rotational eigenvector.
\end{prop}
\begin{proof}
Suppose that $B=Q-R$ and $\check{Q}-\check{R}$ are both rotational eigenvectors.
Then so are $A=rP-(Q+R)$ and $\check{r}\check{P}-(\check{Q}+\check{R})$, since both are orthogonal to Temperley-Lieb and to $Q-R$, $\check{Q}-\check{R}$ respectively. Since $\cQ_4$ has annular multiplicities $*20$, $\omega_A\neq \omega_B$ by \cite[Theorem 5.2.3]{MR2972458}. Hence $\cF(A)=\sigma_A\lambda_A \check{C}$ and $\cF(B)=\sigma_B \lambda_B \check{D}$ where $\check{C},\check{D}$ are distinct elements in the set $\{\check{r}\check{P}-(\check{Q}+\check{R}),\check{Q}-\check{R}\}$. (We avoid using $\check{A},\check{B}$ to avoid confusion.) There are two cases to consider. 

Suppose that $\check{C}=\check{r}\check{P}-(\check{Q}+\check{R})$ and $\check{D}=\check{Q}-\check{R}$. Using $r=\check{r}$, we have
$$
\cF(A)=\sigma_A\left(\frac{\sqrt{r}}{\sqrt{\check{r}}}\right)\check{C}=\sigma_A\check{C}
\text{ and }
\cF(B)=\sigma_B\left(\frac{\sqrt{r(1+\check{r})}}{\sqrt{\check{r}(1+r)}}\right) \check{D}=\sigma_B\check{D}.
$$
Hence we may write $\check{A}=\check{C}$ and $\check{B}=\check{D}$ without confusion. Replacing $\cQ$ with $\cQ_4$, the proofs of Equation \eqref{eqn:QA1}, Theorem \ref{thm:OcneanuQuadruple}, and Corollary \ref{cor:EliminateQ} all hold mutatis mutandis. 
We conclude that $[2]\leq 2$. But $2<\|\Gamma_\pm\|\leq [2]$, a contradiction.

The second case is a bit trickier. Suppose that $\check{C}=\check{Q}-\check{R}$ and $\check{D}=\check{r}\check{P}-(\check{Q}+\check{R})$, where using $r=\check{r}$, we have 
\begin{align*}
\cF(A)&=\sigma_A\left(\frac{\sqrt{r(1+\check{r})}}{\sqrt{\check{r}}}\right) \check{C}=\sigma_A(\sqrt{1+r})\check{C}
\text{ and }\\
\cF(B)&=\sigma_B\left(\frac{\sqrt{r}}{\sqrt{\check{r}(1+r)}}\right) \check{D}=\frac{\sigma_B}{\sqrt{1+r}}\check{D}.
\end{align*}
As in Theorem \ref{thm:OcneanuQuadruple},
$$
\cap_{n+1}(\overline{P'})=\D\frac{\Tr(P')}{\Tr(P)}\check{P}=\left(\frac{[n+2]-r[n]}{1+r}\cdot\frac{(1+r)}{[n+1]}\right)\left(\frac{\jw{n}+\check{D}}{1+r}\right).
$$
Liu's relation applied to $P$ as in Theorem \ref{thm:QuadruplePoint} has the following left hand side after multiplying by $(1+r)^2$:
\begin{align*}
\text{LHS} &= \cF(A+\jw{n})^*\cF(A+\jw{n})\\
&=(1+r)\check{C}^2-\sqrt{1+r}\frac{(\sigma_A+\sigma_A^{-1})}{[n]}\check{C}+T\\
&=-\check{D}-\sqrt{1+r}\frac{(\sigma_A+\sigma_A^{-1})}{[n]}\check{C}+T'
\end{align*}
for some $T,T'\in \TL_{n,-}$.
The right hand side, multiplying by $(1+r)^2$ and ignoring terms in Temperley-Lieb, is given by
$$
\text{RHS}= 
-(1+r)\frac{[n+1]}{[n]}\left(\frac{[n+2]-r[n]}{1+r}\cdot\frac{(1+r)}{[n+1]}\right)\left(\frac{\check{D}}{1+r}\right)
=\frac{[n+2]-r[n]}{[n]}\check{D}.
$$
Since $\check{C},\check{D}$ are orthogonal to Temperley-Lieb and to each other, we must have $\sigma_A+\sigma_A^{-1}=0$, and the coefficients of $\check{D}$ must agree, which means
$$
-1=\frac{[n+2]-r[n]}{[n]}\Longleftrightarrow r=\frac{[n+2]+[n]}{[n]}.
$$
But $\D\Tr(Q)=\Tr(R)=\frac{[n]}{[2]}$ implies $\D\dim(P)=\frac{[n+2]-[n]}{[2]}$ and $\D r=\frac{2[n]}{[n+2]-[n]}$. Hence
$$
\frac{2[n]}{[n+2]-[n]} = r = \frac{[n+2]+[n]}{[n]} \Longleftrightarrow 3[n]^2=[n+2]^2.
$$
It is easy to check that this is impossible when $q>1$.
\end{proof}

\begin{rem}\label{rem:Hypothesis}
By the symmetry of the graphs in $\cQ_4$, it is natural to hypothesize that both $Q-R$ and $\check{Q}-\check{R}$ are rotational eigenvectors, although there is no particular reason this should be the case.
\end{rem}

\begin{cor}\label{cor:GHJ}
For the GHJ 3311 subfactor \cite{MR999799} with principal graphs
$$
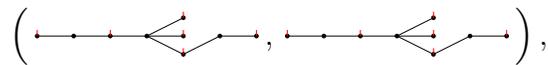
\left(
\bigraph{bwd1v1v1v1p1p1v1x0x0v1duals1v1v1x2x3v1},
\bigraph{bwd1v1v1v1p1p1v1x0x0v1duals1v1v1x2x3v1}
\right),
$$ 
at least one of $Q-R$ or $\check{Q}-\check{R}$ is not a low weight rotational eigenvector.
\end{cor}

\begin{rems}
\mbox{}
\be
\item
In fact, the GHJ 3311 subfactor is the only subfactor whose principal graphs are a translated extension of $\cQ_4$ \cite{MR2993924}.
\item
Corollary \ref{cor:GHJ} also follows from \cite[Theorem 5.10]{1208.3637}, where surprisingly $\check{Q}-\check{R}$ is a rotational eigenvector! 
However, it is convenient to disprove the hypothesis in Remark \ref{rem:Hypothesis} without having to do the substantial work of constructing the subfactor planar algebra.
\ee
\end{rems}

\bibliographystyle{amsalpha}
\bibliography{../bibliography}
\end{document}